\documentclass{article}
\usepackage[utf8]{inputenc}
\usepackage{enumerate}
\usepackage{amssymb, amsmath}
\usepackage{ amsthm, thmtools, thm-restate}
\usepackage{hyperref, cleveref}
\usepackage{tikz}
\usetikzlibrary{arrows}
\usetikzlibrary{decorations.markings}
\usepackage[hmargin=3cm,vmargin=2.5cm]{geometry}
\newcommand{\gcl}{\text{cl}_{\text{g}}}

\newcommand{\tmcl}{\text{cl}_{\text{tm}}}

\newcommand*{\defeq}{\mathrel{\vcenter{\baselineskip0.5ex \lineskiplimit0pt
                     \hbox{\scriptsize.}\hbox{\scriptsize.}}}
                     =}

\declaretheorem[name=Theorem, refname={Theorem, Theorems}, numberwithin=section]{thm}
\declaretheorem[name=Lemma, refname={Lemma,Lemmas}, sibling=thm]{lem}

\declaretheorem[name=Corollary, sibling=thm]{cor}

\theoremstyle{definition}
\declaretheorem[name=Definition, refname={Definition,Definitions}, sibling=thm]{Def}

\title{Reducts of the Generic Digraph}
\date{}

\begin{document}
\maketitle

\vspace*{-20mm}
\begin{abstract}
The generic digraph $(D,E)$ is the unique countable homogeneous digraph that embeds all finite digraphs. In this paper, we determine the lattice of reducts of $(D,E)$, where a structure $\mathcal{M}$ is a reduct of $(D,E)$ if it has domain $D$ and all its $\emptyset$-definable relations are $\emptyset$-definable relations of $(D,E)$. As $(D,E)$ is $\aleph_0$-categorical, this is equivalent to determining the lattice of closed groups that lie in between Aut$(D,E)$ and Sym$(D)$.

\end{abstract}

This paper is a part of a large body of work concerning reducts of first-order structures, where $\mathcal{N}$ is said to be a reduct of $\mathcal{M}$ if all $\emptyset$-definable relations in $\mathcal{N}$ are $\emptyset$-definable in $\mathcal{M}$. Usually, the set-up is that one studies the reducts of some given structure $\mathcal{M}$. When the structure is $\aleph_0$-categorical, this is equivalent to studying the closed subgroups lying between Aut$(\mathcal{M})$ and Sym$(M)$.

The first results in this area were the classification of the reducts of $(\mathbb{Q},<)$ (\cite{cam76}) and of the random graph $\Gamma$ (\cite{tho91}). In \cite{tho96}, Thomas conjectured that all homogeneous structures in a finite relational language have only finitely many reducts. This question remains unsolved and continues to provide motivation for study.  More recent results include the classification of the reducts of $(\mathbb{Q},<,0)$ (\cite{jz08}), of the affine and projective spaces over $\mathbb{Q}$ (\cite{ks13}) and of the $\aleph_0$-dimensional vector space over prime fields (\cite{bks14}). 

A surprising development in this area is the connection with constraint satisfaction in complexity theory, by Bodirsky and Pinsker. This connection is made via clone theory in universal algebra. In order to analyse certain closed clones they developed a Ramsey-theoretic tool, named `canonical functions'. With further developments (\cite{bp11}, \cite{bpt13}), canonical functions now provide a powerful tool in studying reducts, for example, they were used to classify the reducts of the generic partial order (\cite{ppp+}) and of the generic ordered graph (\cite{bpp13}).

In this paper, we determine the lattice of reducts of the generic directed graph, which we denote by $(D,E)$.  For us, a directed graph (or digraph) means a set of vertices with directed edges between them, where we do \emph{not} allow an edge going in both directions.  The generic digraph is the unique countable homogeneous digraph that embeds all finite digraphs. `Homogeneous' means that every isomorphism $f: A \to B$, where $A,B \subset D$ are finite, can be extended to an automorphism of $(D,E)$.

We outline the structure of the paper.  In Section 1, we provide the necessary preliminary definitions and facts about the generic digraph and about reducts. We also comment on some notational conventions that we use. In Section 2, we define the reducts of the generic graph and provide the lattice, $\mathcal{L}$, that these reducts form. The main theorem is that this lattice $\mathcal{L}$ is the lattice of all the reducts of the generic digraph. In Section 3, we describe the reducts in some detail, establishing notation and important lemmas that are used in the rest of the paper. In Section 4, we show that $\mathcal{L}$ is indeed a sublattice of the lattice of reducts. In Section 5, we prove that $\mathcal{L}$ does contain all the reducts of $(D,E)$. The section starts by describing the information that is obtained from the known classifications of the random graph and the random tournament (\cite{ben97}).  We then give the background definitions and results on canonical functions at the start of Section 5.2, and we also carry out the  the combinatorial analysis of the canonical functions in this section. Section 5 ends by using the analysis to complete the proof of the main theorem. In Section 6, we provide a summary and some open questions.

\section{Preliminaries}
\subsection{Notational Conventions}
Structures are denoted by $\mathcal{M},\mathcal{N}$, and their domains are $M$ and $N$ respectively. Sym$(M)$ is the set of all bijections $M \to M$ and Aut$(\mathcal{M})$ is the set of all automorphisms of $\mathcal{M}$. Given a formula $\phi(x,y)$, we use $\phi^*(x,y)$ to denote the formula $\phi(y,x)$. $S(\mathcal{M})$ denotes the space of types of the theory of $\mathcal{M}$. If $f$ has domain $A$ and $\bar{a} \in A$, then $f(a_1,\ldots,a_n) \defeq (f(a_1),\ldots,f(a_n))$. If $\bar{a}$ and $\bar{b}$ are tuples of the same length $n$ we say $\bar{a}$ and $\bar{b}$ are isomorphic, and write $\bar{a} \cong \bar{b}$, to mean that the function $a_i \mapsto b_i$ for all $i$ such that $1 \leq i \leq n$ is an isomorphism.

There will be instances where we do not adhere to strictly correct notational usage, however, the meaning is always clear from the context. For example, we may write `$a \in (a_1,\ldots,a_n)$' instead of `$a=a_i$ for some $i$ such that $1 \leq i \leq n$'. Another example is that we sometimes use $c$ to represent the singleton set $\{ c \}$ containing it.

\subsection{The Generic Digraph}
\begin{Def} \begin{enumerate}[(i)]
\item A directed graph $(V,E)$ consists of a set $V$ and an irreflexive, antisymmetric relation $E \subseteq V^2$. $V$ represents the set of vertices and $E$ represents the set of directed edges, so if $(a,b) \in E$, we visualise it as an edge going out of $a$ and into $b$.  We abbreviate `directed graph' by `digraph'.
\item By an empty digraph we mean a digraph whose edge set is empty.
\item We say that a structure $\mathcal{M}$ is homogeneous if every isomorphism $f:A \to B$, where $A,B$ are finite substructures of $M$, can be extended to an automorphism of $\mathcal{M}$.
\item The generic digraph, which we denote by $(D,E)$, is the unique (up to isomorphism) countable homogeneous digraph that embeds all finite digraphs.
\item $N(x,y) \subset D^2$ will denote the non-edge relation of $(D,E)$, so $N(x,y) \defeq \neg E(x,y) \wedge \neg E^*(x,y)$.
\end{enumerate}
\end{Def}

The fact that the generic digraph exists and is unique follows from the theory of Fra\"{i}ss\'{e} limits and amalgamation classes, originally described in \cite{fra53}. Details and proofs can be found in \cite{hod97}.

The following lemma collects several useful properties of the generic digraph.

\begin{lem} \label{genericdigraph} \begin{enumerate}[(i)]
\item Th$((D,E))$ is $\aleph_0$-categorical and has quantifier elimination.
\item Let $\bar{a},\bar{b} \in D$. If tp$(\bar{a})=$ tp$(\bar{b})$, then there exists an automorphism mapping $\bar{a}$ to $\bar{b}$.
\item The generic digraph $(D,E)$ is the unique, up to isomorphism, countable digraph satisfying the following extension property: for all finite pairwise disjoint subsets $U,V,W \subset D$ there exists $x \in D\backslash (U \cup V \cup W)$ such that $(\forall u \in U) E(x,u)$, $(\forall v \in V) E(v,x)$ and $(\forall w \in W) N(x,w)$.
\item All countable digraphs can be embedded into the generic digraph.
\item Let $A \subseteq D$ and $B=A^c$. Then $(A,E|_A)$ or $(B,E|_B)$ is isomorphic to the generic digraph.
\end{enumerate}
\end{lem}

Remark: Due to the importance of the property in (iii), we give it the name `the extension property'.

Remark: As a result of (ii), there is bijective correspondence between $n$-types and orbits of $n$-tuples.  Given a type $p(\bar{x})$ you obtain the orbit $\{\bar{x} \in D\:$ tp$(\bar{x})=p \}$, and given an orbit $A \subset D^n$ you obtain the type $p(\bar{a})$, where $\bar{a} \in A$. In this light, and as has become customary in modern model theory, we sometimes blur the distinction between a type and the set of tuples that realise that type.

\begin{proof} (i) This is an instance of the more general statement that any countable homogeneous structure in a finite relational language is $\aleph_0$-categorical and has quantifier elimination. See \cite{hod97} for details.

(ii) Since we have quantifier elimination, tp$(\bar{a})=$ tp$(\bar{b})$ implies that $\bar{a} \cong \bar{b}$, so by homogeneiety there is an automorphism that maps $\bar{a}$ to $\bar{b}$.

(iii) We leave this as an exercise for the reader. To show that two countable digraphs which satisfy the extension property are isomorphic, you use a back-and-forth argument. An explanation and examples of back-and-forth arguments can be found in \cite{hod97}.

(iv) This is proved using only the `forth' part of a back-and-forth argument. We sketch the proof. Let $(D',E')$ be a countable digraph, and let $d_1,d_2,d_3,\ldots$ be an enumeration of the elements of $D'$. You then define an embedding of $D'$ into $D$ inductively. The condition that the generic digraph needs to satisfy to ensure that the inductive step works is precisely the extension property.

(v) By (iii), it suffices to show that $(A,E_A)$ or $(B,E_B)$ satisfies the extension property. Suppose for contradiction that both fail the extension property. Let $U_1,V_1,W_1 \subset A$ and $U_2,V_2,W_2 \subset B$ witness this failure. Now let $U=U_1 \cup U_2, V=V_1 \cup V_2$ and $W=W_1 \cup W_2$. These are finite pairwise disjoint subsets of $D$. By (i), we know that $D$ satisfies the extension property, so we can find an appropriate witness $x$ in $D$. Now observe that $x$ is also a witness for $U_1,V_1,W_1$ and for $U_2,V_2,W_2$. But this means we have a contradiction, because $x$ must be in $A$ or in $B$.
\end{proof}

\subsection{Reducts}
Let $\mathcal{M}$ be a structure on domain $M$. A relation $P \subseteq M^k$ is $\emptyset$-definable in $\mathcal{M}$ if there exists a formula $\phi(x_1,\ldots,x_k)$ in the language of $\mathcal{M}$ such that $P=\{(x_1,\ldots,x_k) \in M^k: \mathcal{M} \models \phi(x_1,\ldots,x_k)\}$.

Let $\mathcal{M}$ and $\mathcal{N}$ be two structures on the same domain $M$. We say that $\mathcal{N}$ is a reduct of $\mathcal{M}$ if for all $k \in \mathbb{N}$ and all relations $P \subset M^k$, if $P$ is $\emptyset$-definable in $\mathcal{N}$ then $P$ is $\emptyset$-definable in $\mathcal{M}$.  We say $\mathcal{N}$ is a proper reduct of $\mathcal{M}$ if $\mathcal{N}$ is a reduct of $\mathcal{M}$ and $\mathcal{N} \neq \mathcal{M}$.

The question that is answered here is: What are the reducts of the generic digraph?  For this question to be meaningful an important caveat is required, which is that if two structures are both reducts of each other - which implies that they are (first-order) interdefinable - we regard them as being equal. This is the reason you will find the phrase `up to interdefinability' used in the literature. For the sake of conciseness, we choose to avoid this phrase with the understanding that we will always consider two reducts that are interdefinable to be equal.

An important fact about the reducts of a fixed structure $\mathcal{M}$ is that they form a lattice, where $\mathcal{N} \leq \mathcal{N'}$ if $\mathcal{N}$ is a reduct of $\mathcal{N'}$. The top element is always the original structure $\mathcal{M}$ and the bottom element is the trivial structure $(M,=)$. The meet (respectively join) of two structures $\mathcal{N}$ and $\mathcal{N'}$ will be the structure whose named relations are precisely the $\emptyset$-definable relations that are definable in both (respectively in at least one of) $\mathcal{N}$ and $\mathcal{N'}$.  Intuitively, the meet contains the intersection of the information in the two structures, and the join contains the union of the information.  In addition to determining what the reducts of the generic digraph are, we also determine how they relate in this lattice.

There is a second, closely related notion of a reduct known as a group reduct. We say that $\mathcal{N}$ is a group reduct of $\mathcal{M}$ if Aut$(\mathcal{N}) \geq$ Aut$(\mathcal{M})$. The group reducts of a fixed structure $\mathcal{M}$ form a lattice via the usual inclusion operation; the bottom element is Aut$(\mathcal{M})$ and the top element is always Sym$(D)$.

As a consequence of the Engeler--Ryll-Nardzewski--Svenonius theorem (see \cite{hod97}), if $\mathcal{M}$ is $\aleph_0$-categorical then the lattice of reducts is anti-isomorphic to the lattice of group-reducts. In one direction, a reduct $\mathcal{N}$ is mapped to its automorphism group Aut$(\mathcal{N})$. In the other direction, given a group reduct $G$ you let $\mathcal{N}$ be the structure whose $n$-ary relations are the orbits of the action of $G$ on $M^n$ (where for all $g \in G, \bar{x} \in M^n, g\cdot \bar{x}=g(\bar{x})$).  In this light, we often use the word `reduct' to refer to either notion, with the meaning being clear from the context.

Furthermore, group reducts of $\mathcal{M}$ can be described purely in terms of permutation group theory, without reference to structures. To do this, we need to consider the topological structure of Sym$(M)$.  There are two ways of defining the topology. The first is to say that the topology on Sym$(M)$ is the subspace topology of $M^M$, where $M^M$ has the product topology, where $M$ is given the discrete topology.  The second (equivalent) way is to say what it means for $F \subseteq$ Sym$(M)$ to be closed: We say that $g \in$ Sym$(M)$ is in the closure of $F$ if for all finite $A \subset M$, there exists $f \in F$ such that $f(a)=g(a)$ for all $a \in A$. Then, $F$ is closed if $F$ is equal to the closure of itself.

It is a central fact in permutation group theory that $G \leq$ Sym$(M)$ is closed if and only if there exists a structure with domain $M$ such that $G$ is its automorphism group.  Thus, the group reducts of $\mathcal{M}$ are exactly the closed groups $G \leq$ Sym$(M)$ that contain Aut$(\mathcal{M})$.

From the above discussion, since $(D,E)$ is $\aleph_0$-categorical, the task of determining its reducts is the same as determining its group reducts, which in turn is the same as determining the closed groups $G$ where Aut$(D,E) \leq G \leq$ Sym$(D)$.


\section{Defining the Reducts}
There are two ways of defining reducts, corresponding to the two different notions of reducts. On the permutation group theoretic side, you can define a reduct by adding a function $f \in$ Sym$(D)$ to Aut$(D,E)$, then closing under group operations and closing under the topology. By considering the model theoretic view, you first define a relation, $P$ say, and define the reduct to be the automorphism group of $(D,P)$. In view of this, we establish some notation:

\begin{enumerate}[(i)]
\item Let $G$ be a topological group (e.g. Sym($D$)). For $F \subseteq G$, let $\langle F \rangle$ denote the smallest closed subgroup of $G$ containing $F$. For brevity, when it is clear we are discussing reducts of $(D,E)$, we may abuse notation and write $\langle F \rangle$ to mean $\langle F \cup $ Aut$(D,E) \rangle$.

\item Let $G$ be a group. For $F \subseteq G$, let $\gcl(F)$\footnote{where $\gcl$ stands for `group closure'} denote the smallest subgroup of $G$ containing $F$. As above, we may abuse notation where it is clear we are discussing supergroups of Aut($D,E$).
\end{enumerate}

We begin by showing in the next few lemmas that three particular functions $-,sw$ and $rot$ exist. These functions will give us the three reducts $\langle - \rangle, \langle sw \rangle$ and $\langle rot \rangle$.

\begin{lem} There exists $f:D \to D$ such that for all $x,y \in D$, $E(f(x),f(y))$ iff $E(y,x)$.
\end{lem}
Remark. For the rest of this article, we fix such a function and denote it by $-$.
\begin{proof} The idea is to define a structure $(D,E')$ which is isomorphic to $(D,E)$, in such a way that any isomorphism $f:D \to D$ witnessing this fact has the desired property. For this lemma, we let $E'(x,y) = E^*(x,y) \defeq E(y,x)$. We need to show that $(D,E')$ is  isomorphic to $(D,E)$. 

By \autoref{genericdigraph}, it suffices to show that $(D,E')$ satisifes the extension property. So let $U,V,W$ be finite disjoint subsets of $D$. By the definition of $E'$, we need to find $x \in D\backslash (U \cup V \cup W)$ such that $\forall u \in U, E(u,x), \forall v \in V, E(x,v)$ and $\forall w \in W, N(x,w)$.  This is simply the extension property for $(D,E)$ with the role of $U$ and $V$ swapped, so we know such an $x$ exists (again by \autoref{genericdigraph}). Thus, $(D,E)$ and $(D,E')$ are isomorphic.

Now let $f:D \to D$ be an isomorphism from $(D,E)$ to $(D,E')$ to complete the proof.
\end{proof}

\begin{lem} \label{swexists} Let $a \in D$. Then there exists $f:D \to D$ such that
\[ E(f(x),f(y)) \text{ if and only if } \begin{cases}
  E(x,y) \text{ and } x,y \neq a, \text{ OR,}\\
  E^*(x,y) \text{ and } x=a \vee y=a
\end{cases}
\]
\end{lem}
Remark: For the rest of this article, we fix such a function and denote it by $sw$.

\begin{proof} As in the previous lemma, the idea is to find an appropriate structure $(D,E')$ isomorphic to $(D,E)$. For this, we define $E'(x,y)$ as follows: \[ E'(x,y) \defeq \begin{cases}
  E(x,y),& \text{if } x,y \neq a\\
  E^*(x,y),& \text{otherwise}
\end{cases}
\]

As before, \autoref{genericdigraph} tells us that we need to establish the extension property for $(D,E')$. Let $U,V,W \subset D$ be finite and pairwise disjoint. This time the proof splits into three cases.

Case 1: $a \in U$. Let $U'=U\backslash \{a\}$ and $V'=V \cup \{a\}$. Then the extension property of $(D,E)$ applied to $U',V',W$ gives an appropriate $x$.

Case 2: $a \in V$. Let $U'=U \cup \{a\}$ and $V'=V \backslash \{a\}$. Then again the extension property of $(D,E)$ gives us an appropriate $x$.

Case 3: $a \in W$ or $a \notin U \cup V \cup W$. Then applying the extension property of $(D,E)$ gives us an appropriate $x$, without needing to modify $U,V$ or $W$.

Thus, $(D,E')$ satisfies the extension property, and hence is isomorphic to $(D,E)$. We end by letting $f$ witness this isomorphism.
\end{proof}

\begin{lem}
Let $a \in D$. Then there exists $f:D \to D$ such that
\[ E(f(x),f(y)) \text{ if and only if } \begin{cases}
  x,y \neq a \text{ and } E(x,y)\\
  x=a \text{ and } N(x,y)\\
  y=a \text{ and } E^*(x,y)
\end{cases}
\]
\end{lem}
Remark: For the rest of this article, we fix such a function and denote it by $rot$.

Remark: In words, $rot$ sends edges going out of $a$, to edges going into $a$, to non-edges, to edges going out of $a$.

\begin{proof} Use the same strategy as for $-$ and $sw$. 
\end{proof}

\begin{Def} \begin{enumerate}[(i)]
\item We let $\Gamma = (D,E_{\Gamma})$, where $E_{\Gamma} \defeq E(x,y) \vee E^*(x,y)$. $\Gamma$ is a graph and, as will be proved later, is in fact (isomorphic to) the random graph.
\item We let $-_{\Gamma} \in$ Sym$(D)$ be a function which interchanges the sets of edges and non-edges in $\Gamma$.
\item Let $a \in D$. We let $sw_{\Gamma} \in$ Sym$(D)$ be a function which interchanges the sets of edges and non-edges adjacent to $a$, and preserves all other edges and non-edges.
\end{enumerate}
\end{Def}
Remark: $(D,E_{\Gamma})$ is inter-definable with $(D,N)$, where $N(x,y)$ says that $xy$ is a non-edge.

We now have all the background definitions necessary to state the main theorem:

\begin{thm} \label{maintheorem} The reducts of $(D,E)$ are given by the following lattice, which we call $\mathcal{L}$:
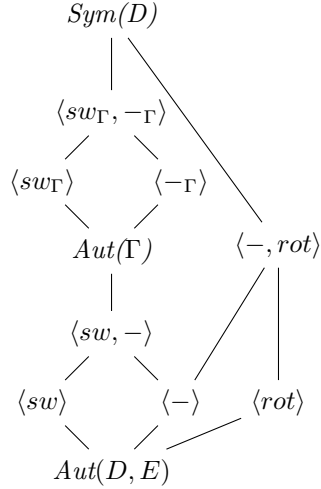
\begin{center}
\begin{tikzpicture}[node distance=1.3cm, auto]
  \node (D) {Aut$(D,E)$};
  \node (sw) [above left of=D] {$\langle sw \rangle$};
  \node (glo) [above right of=D] {$\langle - \rangle$};
  \node (sup) [above right of=sw] {$\langle sw,- \rangle$};
  \node (Ra) [above of=sup, yshift=-2mm] {Aut($\Gamma$)};
  \node (sw2) [above left of=Ra] {$\langle sw_{\Gamma} \rangle$};
  \node (glo2) [above right of=Ra] {$\langle -_{\Gamma} \rangle$};
  \node (sup2) [above right of=sw2] {$\langle sw_{\Gamma},-_{\Gamma} \rangle$};
  \node (top) [above of=sup2] {Sym(D)};
  \node (rot) [right of=glo] {$\langle rot \rangle$};
  \node (rot2)[above of=rot, yshift=8mm] {$\langle -,rot \rangle$};
  \draw[-] (D) to node {} (sw);
  \draw[-] (D) to node {} (glo);
  \draw[-] (sw) to node {} (sup);
  \draw[-] (glo) to node {} (sup);
  \draw[-] (sup) to node {} (Ra);
  \draw[-] (Ra) to node {} (sw2);
  \draw[-] (Ra) to node {} (glo2);
  \draw[-] (sw2) to node {} (sup2);
  \draw[-] (glo2) to node {} (sup2);
  \draw[-] (sup2) to node {} (top);
  \draw[-] (D) to node{}(rot);
  \draw[-] (rot) to node{}(rot2);
  \draw[-] (glo) to node{}(rot2);
..\draw[-] (rot2) to node{}(top);
\end{tikzpicture}
\end{center}
\end{thm}

This theorem can be split into two main claims. The first is that $\mathcal{L}$ is a sublattice of the reducts of $(D,E)$ (so for example one needs to show that the meets and joins are correct). The second claim is that $\mathcal{L}$ is in fact the whole lattice - that there are no other reducts. The second claim is the more interesting claim, and requires more work to prove.


\section{Understanding the reducts}
The purpose of this section is twofold.  The first is to establish conditions for an unknown reduct $G$ of $(D,E)$ to be equal to or to contain particular elements of $\mathcal{L}$ - these lemmas will be used throughout the article. The second is to provide familiarity with the reducts, without which the article may be more difficult to understand.

The first few lemmas will provide a concrete description of the three groups $\langle sw \rangle, \langle - \rangle$ and $\langle rot \rangle$. The way we do this is by comparing how two functions behave, via the following definition.

\begin{Def} \label{behaves}
Let $f,g: D \to D$ and $A \subset D$. We say \emph{$f$ behaves like $g$ on $A$} if for all finite tuples $\bar{a} \in A$, $f(\bar{a})$ is isomorphic (as a finite digraph) to $g(\bar{a})$. If $A=D$, we simply say \emph{$f$ behaves like $g$}.
\end{Def}

Example. All automorphisms of $(D,E)$ behave like the identity $id:D \to D$. Conversely, all $f \in Sym(D)$ which behave like $id$ are automorphisms.

\textbf{Important Remark.} If $f:D \to D$ is any function and $g \in$ Aut$(D,E)$, then $h \defeq g \circ f$ behaves like $f$. The converse it also true: if $h$ behaves like $f$, then there is $g \in$ Aut$(D,E)$ such that $h= g \circ f$.

Before continuing, we note the following useful fact. If a bijection $f$ and its inverse both preserve a definable relation $P$, then the group $\langle$Aut$(D,E) \cup \{f\} \rangle$ also preserves $P$. This follows straightforwardly by unravelling the definitions, and doing this would be a worthwhile exercise for the reader first encountering these notions.

We start with the simplest of the three groups, $\langle - \rangle$.

\begin{lem} \label{understanding-}Let $f \in Sym(D)$. Then $f \in \langle - \rangle \backslash$ Aut$(D,E) \Leftrightarrow f$ behaves like $-$.
\end{lem}
\begin{proof} ``$\Leftarrow$''. We need to show that $f \in \langle - \rangle$. Consider the function $g \defeq - \circ f$. It is easy to see that for all tuples $\bar{a} \in D$, $g(\bar{a})$ is isomorphic to $\bar{a}$. This means that $g$ behaves like $id$, so $g \in$ Aut$(D,E)$. Hence, $f = -^{-1} \circ g \in \gcl(-) \subseteq \langle - \rangle$, so we are done.

``$\Rightarrow$''. $-$ and $-^{-1}$ preserve the weakened edge relation $E_w (x,y;a,b) \defeq E(x,y) \leftrightarrow E(a,b)$, so $\langle - \rangle$ must also preserve $E_w$. In addition, the non-edge relation $N(x,y)$ is definable from $E_w$: $N(x,y) \Leftrightarrow \forall a,b  (E_w(x,y;a,b) \leftrightarrow E_w(y,x;a,b))$; hence, $\langle - \rangle$ also preserves non-edges.

Now suppose $f$ does not behave like $-$ on $D$ -  we want to show $f \notin \langle - \rangle \backslash$ Aut$(D,E)$. If $f$ is an automorphism, then we're trivially done, so assume $f \notin$ Aut$(D,E)$. If $f$ does not preserve non-edges, then we are also done by the previous paragraph; so assume $f$ does preserve non-edges. The only possibility that remains is that there are edges $ab,cd \in D$ such that $E(f(a),f(b))$ and $\neg E(f(c),f(d))$. This means that $E_w(a,b;c,d)$ and $\neg E_w(f(a,b;c,d))$, i.e. that $f$ does not preserve $E_w$. Thus, $f \notin \langle - \rangle$, as required.
\end{proof}


Next we look at $\langle sw \rangle$. To do this we need some notation. For $A \subset D$, we let $sw_A: D \to D$ denote a function which behaves like $id$ on $A$ and $A^c$, and which switches the direction of all edges between $A$ and $A^c$. For example, $sw=sw_a$ for some $a \in D$, and, $sw_{\emptyset}$ is just an automorphism.  The fact that $sw_A$ exists for all $A \subseteq D$ follows from the fact that all countable digraphs are embeddable in the generic digraph (\autoref{genericdigraph}). However, $sw_A$ cannot be a bijection for all $A \subset D$. This is because the image of the generic digraph on applying $sw_A$ may not be isomorphic to the digraph. For example, if you let $A=\{x \in D: E(a,x)\}$ where $a$ is some element of $D$, then $sw_A(a)$ will not have any outward edges. However, there are many subsets of $A$ for which $sw_A$ can be a bijection. For example, if $A \subset D$ is finite, one checks that the digraph obtained by switching with respect to a $A$ satisfies the extension property, so it is isomorphic to the generic digraph.

A big idea in the next lemma is this: Let $a_1,\ldots,a_n$ be distinct elements of $D$. Then $sw_{a_1} \circ \ldots \circ sw_{a_n}$ behaves like $sw_A$, where $A=\{a_1,\ldots,a_n\}$. The problem with this idea is that, as stated, it is false: this is because the points $a_1,\ldots,a_n$ will not necessarily be fixed by each of the $sw_{a_i}$'s. \emph{Do} however keep this idea in mind, as it provides the intuition for (parts of) the lemma.

\begin{lem} \label{understandingsw}
\begin{enumerate}[(i)]
\item $\gcl(sw) = \{f \in$ Sym$(D): f$ behaves like $sw_A$, for some finite $A \subset D\}.$
\item For all $A \subseteq D$, if $sw_A \in$ Sym($D)$ then $sw_A \in \langle sw \rangle$.
\item For all proper non-empty $A \subset D$, if $sw_A \in$ Sym$(D)$ then $\langle sw_A \rangle = \langle sw \rangle$.
\item $\langle sw \rangle = \{ f \in$ Sym$(D): f$ behaves like $sw_A$, for some $A \subseteq D\}.$
\end{enumerate}
\end{lem}

\begin{proof} For all of this proof, let $a \in D$ be the point such that $sw = sw_a$.

(i) RHS $\subseteq$ LHS. From the important remark above, in order to show that every $f$ which behaves like $sw_A$ is in $\gcl(sw)$, it suffices to show that $sw_A \in \gcl(sw)$.

First, we show that $sw_{a'} \in \gcl(sw)$, for all $a' \in D$. This is easy: let $g \in$ Aut$(D,E)$ map $a'$ to $a$. Then $sw \circ g \in \gcl(sw)$ and $sw \circ g$ behaves like $sw_{a'}$. Thus, again by the important remark, $sw_{a'} \in \gcl(sw)$.

Now let $A=\{a_1,\ldots,a_n\} \subset D$. We start by letting $h_1 = sw_{a_1}$. Then let $h_2 = sw_{h_1(a_2)} h_1$ - observe that $h_2$ behaves like $sw_{\{a_1,a_2\}}$. Next let $h_3 = sw_{h_2(a_3)} h_2$ - $h_3$ behaves like $sw_{\{a_1,a_2,a_3\}}$. Continuing, we obtain $h_n$ which behaves like $sw_A$. By construction, $h_n \in \gcl(sw)$ and so by the important remark $sw_A \in \gcl(sw)$, as required.

LHS $\subseteq$ RHS: Any $f \in \gcl(sw)$ can be written as $g_n sw^{\epsilon_n} \ldots g_1 sw^{\epsilon_1} g_0$, where the $g_i$ are automorphisms and $\epsilon_i \in \{1,-1\}$. Since $sw^{-1}$ behaves like $sw$, it is equal to $g\circ sw$, for some $g \in$ Aut$(D,E)$, so without loss, $\epsilon_i=1$ for all $i$.

We prove by induction on $n$ that $f$ behaves like $sw_A$ for some finite $A$. In the base case, $f =g_0$ which behaves like $sw_{\emptyset}$. So assume that $f' \defeq g_n sw \ldots g_1 sw g_0$ behaves like $sw_A$ for some finite $A$; we consider $f=g_{n+1}sw f'$. Let $a' = f'^{-1}(a)$. If $a' \notin A$, then $f$ behaves like $sw_{A \cup \{a'\}}$. If $a' \in A$, then $f$ behaves like $sw_{A \backslash \{a'\}}.$ In both cases, we have what we want, thus completing the proof.

(ii) Let $A \subseteq D$ and $sw_A \in$ Sym$(D)$. We need to show that for all finite tuples $\bar{d} \in D$, there exists $g \in \gcl(sw)$ such that $g(\bar{d})=sw_A(\bar{d})$.

Let $A' = A \cap \bar{d}$. $A'$ is finite, so by part (i), $sw_{A'} \in \gcl(sw)$. Now, $sw_{A'}(\bar{d})$ is isomorphic to $sw_A(\bar{d})$, so by homogeneity let $h \in$ Aut$(D,E)$ map $sw_{A'}(\bar{d})$ to $sw_A(\bar{d})$. Letting $g=hsw_{A'}$ finishes the proof.

(iii) Part (ii) tells us that $\langle sw_A \rangle \subseteq \langle sw \rangle$. To show the other direction, it suffices to show that $sw \in \langle sw_A \rangle$.

So let $A \subset D$ be such that $A$ and $A^c$ are non-empty. By unravelling the definitions, we need to prove the following: For all $a_1,\ldots a_n \in D$, there exist $b_1,\ldots,b_n \in D$ such that $\bar{a} \cong \bar{b}$, and $A \cap \bar{b} = \{b_1\}$ or $\{b_2,\ldots,b_n\}$.

If $A$ is finite, we let $b_1$ be any element of $A$ and find the remaining $b_2,\ldots,b_n$ by homogeneity. By the same reasoning, we are done if $A^c$ is finite. Hence, assume that $A$ is infinite and co-infinite.

We prove the result by induction on the length $n$ of the tuple $\bar{a}$. The base case $n=1$ is trivial - simply let $b_1$ be any element of $A$. Now let $(a_1,\ldots, a_{n+1})$ be any tuple of length $n+1$. By the inductive hypothesis, we can find $(b_1,\ldots,b_n)$ isomorphic to $(a_1,\ldots,a_n)$ where $A \cap \bar{b} = \{b_1\}$ or $\{b_2,\ldots,b_n\}$. Without loss, we may assume that $A \cap \bar{b} = \{b_1\}$: the argument is symmetric in the other case.

If we find $x \in A^c$ such that $(b_1,\ldots,b_n,x) \cong \bar{a}$, then we are done, so from now on assume that $(b_1,\ldots,b_n,x) \cong \bar{a}$ implies $x \in A$. $(*)$

Now consider a tuple $(c_1,\ldots,c_{n+1})$ satisfying the following:
\begin{itemize}
\item $c_1$ is some element of $A^c \backslash \{b_2,\ldots,b_n\}$.
\item $\bar{c} \cong \bar{a}$.
\item For each $2 \leq i \leq n+1$, $(b_1,\ldots,b_n,c_i) \cong \bar{a}$
\end{itemize}

The first condition can be satisfied as $A^c$ is infinite. The latter two conditions can be satisfied because $(D,E)$ is homogeneous. By $(*), c_2,\ldots,c_{n+1} \in A$. So $(c_1,\ldots,c_{n+1})$ satisfies all the conditions that we want, completing the induction and hence the proof.

(iv) By part (ii), we have RHS $\subseteq$ LHS. To prove the other direction, we find a relation $P$ that all functions in $\langle sw \rangle$ preserve, and show that if $f$ does not behave like $sw_A$ for any $A$, then $f$ does not preserve $P$.

The relation is: \[
\begin{aligned}
P(x,y,z) \defeq &(E(x,y) \wedge E(y,z) \wedge E(x,z))\\
            \vee &(E^*(x,y) \wedge E^*(y,z) \wedge E(x,z))\\
            \vee &(E^*(x,y) \wedge E(y,z) \wedge E^*(x,z))\\
            \vee &(E(x,y) \wedge E^*(y,z) \wedge E^*(x,z))
\end{aligned}
\]

`Motto': A function preserves $P$ if for all tournaments on three vertices, it switches an even number of edges.

To show that $\langle sw \rangle$ preserves $P$, it suffices to show that $sw$ preserves $P$. This is easy to see. First, $sw$ clearly preserves non-edges. Second, given any three vertices which form a tournament, either $sw$ does not switch any of the edges, or, it switches the direction of precisely two edges (and it would be those two edges which are adjacent to $a$).

Now let $f \in$ Sym$(D)$ be a function which does not behave like $sw_A$ for any $A \subseteq D$. Define a partition of $D$ into subsets as follows: \begin{itemize}
\item Let $A_0 = \{a_0\}$, where $a_0$ is any element of $D$.
\item Let $A_1=\{x \in D: x$ is adjacent to $a_0$ and $f$ does not switch this edge$\}$
\item Let $B_1=\{x \in D: x$ is adjacent to $a_0$ and $f$ switches this edge$\}$
\item Let $A_2=\{x \in D:$ there is an edge between $A_1$ and $x$ that is not switched by $f \}$
\item Let $B_2=\{x \in D:$ there is an edge from $A_1$ to $x$ and all edges between $A_1$ and $x$ are switched by $f \}$
\item Let $A_3=\{x \in D:$ there are no edges between $A_1$ and $x$ and there is an edge between $B_1$ and $x$ switched by $f \}$
\item Let $B_3=\{x \in D:$ there are no edges between $A_1$ and $x$ and all edges between $B_1$ and $x$ are not switched by $f \}$.
\end{itemize}

By construction, these sets are pairwise disjoint. The fact their union equals $D$ follows from the fact that the maximum path length in the generic digraph is two.

The idea behind defining these sets is that \emph{if} $f$ behaved like $sw_A$, then this procedure would find $A$ for us ($A$ would be the union of the $A_i$'s or the union of the $B_i$'s). In this light, let $A=A_0 \cup \ldots \cup A_3$ and $B$ be its complement. By assumption, $f$ does not behave like $sw_A$. What is left in the proof is simply a matter of case checking: we look at the possible reasons $f$ could not behave like $sw_A$ and show in each one that $f$ does not preserve $P$. 

Case 1a: There exists an edge $x,y \in A_1$ that is switched by $f$. Then consider the tournament $(a_0,x,y)$ - $f$ switches exactly one edge, so by the motto $f$ does not preserve $P$.

Case 1b: There exists an edge $x,y \in B_1$ that is switched by $f$. Then $f$ switches all three edges of $(a_0,x,y)$, so $f$ does not preserve $P$.

Case 1c: There exists an edge $x,y \in B_2$ switched in $f$. Let $z$ be any element of $A_1$. Then $f$ switches one edge in $(x,y,z)$.

Case 1d: There exists an edge $x,y \in A_2$ switched by $f$. By definition of $A_2$ there is an $x' \in A_1$ such that $f$ does not switch the edge $x'x$, and there is a corresponding $y'$ for $y$. If $x'=y'$, then we get that $f$ switches one edge in $(x,y,x')$. If $x' \neq y'$, consider the tournament $(x,y,x',y')$.  Now consider any element $z \in D$ such that there is an edge between $z$ and all the vertices $x,y,x',y'$. No matter what $f$ does to these edges, we will be able to find a tournament on three vertices on which $f$ switches an odd number of edges.  For example, if $f$ switched all the edges between $z$ and $x,y,x',y'$, then look at $(x,y,z)$. \footnote{Note that what happens between between $x$ and $y'$ and between $x'$ and $y$ does not matter.}

If there is an edge inside $B_3$ that is switched, then look at any point in $B_1$. If there is an edge inside $A_3$ that is switched, use a similar argument as in Case 1d but using $B_1$ in place of $A_1$. We have now dealt with all edges whose points lie in the same part.

Case 2a: There is an edge $xy$ between $A_1$ and $B_1$ not switched by $f$. Then $f$ switches direction of one edge of $(a_0,x,y)$.

Case 2b: There is an edge $xy$ between $A_1$ and $A_2$ which is switched by $f$. Let $y' \in A_1$ be such that $yy'$ is an edge not switched by $f$. Then look at $(x,y,x',a_0)$ and use the argument in Case 1d.

We have now dealt with all edges containing a point in $A_1$.

Case 2c: There is an edge $xy$ between $A_2$ and $A_3$ which is switched. Let $x' \in A_1$ be adjacent to $x$, and let $y' \in B_1$ be adjacent to $y$ such that $yy'$ is switched. Then consider $(a_0,x,y,x',y')$ and continue as in Case 1d.

If there is an edge between $A_2$ and $B_1$ that is not switched, use Case 1d. Dealing with an edge between $A_2$ and $B_2$ that is not switched is straightforward. If there is an edge between $A_2$ and $B_3$ that is not switched, then continue as in Case 2c. We have now dealt with all edges containing a point in $A_2$.

Case 2d: There is an edge $xy$ between $B_1$ and $B_2$ switched by $f$. Let $z \in A_1$ be a vertex adjacent to $y$. Consider the tournament $(x,y,z,a_0)$, and use the same argument as in Case 1d.

Case 2e: There is an edge $xy$ between $B_1$ and $A_3$ not switched by $f$. Use an argument similar to Case 2b.

This deals with all the edges containing a point in $B_1$.

Case 2f: The case where there is an edge between $A_3$ and $B_3$ which is not switched is straightforward. If there is an edge $xy$ between $A_3$ and $B_2$ that is not switched, let $x' \in B_1$ be such that $xx'$ is an edge that is switched, and $y' \in A_1$ be an edge that is switched. Then look at $(x,y,x',y',a_0)$ and continue as in Case 1d.

Case 2g: There is an edge between $B_2$ and $B_3$ which is switched. Continue as in 2f.

This completes all the cases, and thus the proof.
\end{proof}

\textbf{Remark.} The proof of part (iv) also shows that $\langle sw \rangle = \{f \in Sym(D): f$ preserves $P(x,y,z) \}$. Due to the importance of this relation, we give it a definition.

\begin{Def} Let $P_{sw}(x,y,z)$ be the 3-ary relation $P$ from the proof above.\end{Def}


The next reduct we analyse is $\langle rot \rangle$. The ideas and proofs are analogous to those of $\langle sw \rangle$ so for the sake of conciseness, we will not go into as much detail and may only sketch the idea for some proofs.

\textbf{Notation.} For what follows, $A,B,C \subseteq D$ are pairwise disjoint. For the ordered pair $(A,B)$, an \emph{outward edge} is an edge going from $A$ to $B$ and an \emph{inward edge} is one going from $B$ to $A$. We say $f$ behaves like $rot$ between $(A,B)$\footnote{We may also write `between $A$ and $B$'} if $f$ maps outward edges to inward edges to non-edges to outward edges. We let $rot_{A,B,C}$ be a function $D \to D$ which behaves like $id$ on $A,B$ and $C$ and behaves like $rot$ between $(A,B), (B,C)$ and $(C,A)$. If $C=(A \cup B)^c$, we just write $rot_{A,B}$. If in addition $C=\emptyset$, so that $B=A^c$, we just write $rot_A$.

Simple observations: $rot=rot_a$ for some $a \in D$. If $f$ behaves like $rot_{A,B,C}$ then $f^2$ and $f^{-1}$ behave like $rot_{C,B,A}$, and $f^3$ behaves like $id$. $rot_{B,C,A}$ and $rot_{C,A,B}$ both behave like $rot_{A,B,C}$.

As we did for $sw$, we describe a key idea in the following lemma. Let $a_1,\ldots, a_n,b_1,\ldots,b_m \in D$ be distinct elements. The idea is that $rot_{a_1}^2 \ldots rot_{a_n}^2 rot_{b_1} \ldots rot_{b_m}$ behaves like $rot_{A,B}$ where $A=\{a_1,\ldots a_n\}$ and $B=\{b_1,\ldots,b_n\}$. As before, this is not true as stated because the $a$'s and $b$'s are not fixed points of the functions involved.

\begin{lem} \label{understandingrot}
\begin{enumerate}[(i)]
\item $\gcl(rot) = \{f \in $Sym$(D): f$ behaves like $rot_{A,B}$ where $A,B$ are finite$\}.$
\item For any disjoint $A,B \subseteq D$, if $rot_{A,B} \in$ Sym$(D)$ then $rot_{A,B} \in \langle rot \rangle$.
\item Let $A,B$ be proper disjoint subsets of $D$ such that at least one of $A$ or $B$ is non-empty. If $rot_{A,B} \in$ Sym$(D)$, then $\langle rot_{A,B} \rangle = \langle rot \rangle$.
\item $\langle rot \rangle = \{f \in $Sym$(D): f$ behaves like $rot_{A,B}$ where $A,B$ are disjoint subsets of $D\} $.
\end{enumerate}
\end{lem}
\begin{proof}
For this proof, let $a \in D$ be the point such that $rot=rot_a$.

(i) RHS $\subseteq$ LHS. It suffices to show that $rot_{A,B} \in \gcl(rot)$. We start by showing that $rot_{a'} \in \gcl(rot)$ for all $a' \in D$. This is easy: let $g \in$ Aut$(D,E)$ map $a'$ to $a$ then consider $rot \circ g$. 

For the general case, let $A=\{a_1,\ldots,a_n\}$ and $B=\{b_1,\ldots,b_m\}$. The idea is to rotate twice about each element of $A$ and rotate once about each element of $B$ - we leave the details to the reader.

LHS $\subseteq$ RHS. Any $f \in \gcl(rot)$ can be written in the form $g_n rot^{\epsilon_n} \ldots g_1 rot^{\epsilon_1} g_0$ where for all $i, \epsilon_i \in \{1,-1\}$. Since $rot^{-1}$ behaves like $rot^2$, we can assume that $\epsilon_i=1$ for all $i$. We prove by induction on $n$ that there exist finite disjoint $A,B \subset D$ such that  $f$ behaves like $rot_{A,B}$. 

The base case $n=0$ is trivial, so assume that we know $h=g_{n-1} rot \ldots g_1 rot g_0$ behaves like $rot_{A,B}$ for finite $A,B$, and we consider $f=g_n rot\,h$. There are three cases depending on $a' \defeq h^{-1}(a)$. If $a' \notin A \cup B$, then $f$ behaves like $rot_{A,B \cup \{a'\}}.$ If $a' \in B$, then $f$ behaves like $rot_{A \cup \{a'\}, B \backslash \{a'\}}$. Lastly, if $a' \in A$, then $f$ behaves like $rot_{A \backslash \{a'\},B}$. This completes the induction and hence the proof.

(ii) This is straightforward - just unravel the definitions and use part (i).

(iii) Let $A,B \subseteq D$ be as described in the lemma, and let $C= (A \cup B)^c$. By (ii), we know that LHS $\subseteq$ RHS. To show the other direction, it suffices to show that $rot$ or $rot^{-1} \in \langle rot_{A,B} \rangle$. 

If one of $A,B$ or $C$ is empty, then we are done by imitating the corresponding argument for $\langle sw \rangle$. So assume $A,B$ and $C$ are all non-empty. Now, if $(B \cup C, E_{B \cup C}$ is isomorphic to the generic graph, then we can ignore $A$ and treat it as if it were empty, so again we can imitate the argument from the switching case to get the result.

Hence, assume that $B \cup C$ is not isomorphic to the generic digraph. This means there exist finite, pairwise disjoint $U,V,W \subset B \cup C$ such that if $x \in D$ satisfies $\phi(x) \defeq (\forall u \in U E(u,x)) \wedge (\forall v \in V E^*(v,x)) \wedge (\forall w \in W N(w,x))$, then $x \in A$.

Suppose that there exists $c \in C \cap (U \cup V \cup W)$. We will show that for all $(d_1,\ldots,d_n) \in D$, there exists $a_2,\ldots,a_n \in A$ such that $(c,a_2,\ldots,a_n) \cong (d_1,\ldots,d_n)$. By unravelling definitions, it is easy to see that this is sufficient to show that $rot \in \langle rot_{A,B} \rangle$. So, let $(d_1,\ldots,d_n) \in D$. Then let $(a_2,\ldots,a_n) \in D$ be such that $D \models \phi(a_2),\ldots,\phi(a_n)$ and $(c,a_2,\ldots,a_n) \cong (d_1,\ldots,d_n)$. Such $a_i$ exist by the homogeneity of $(D,E)$. Since $\phi(a_i)$ for all $i$, $(a_2,\ldots,a_n)$ has to be in $A$, as required, so $rot \in \langle rot_{A,B} \rangle$.

Now suppose that $C \cap (U \cup W \cup V) = \emptyset$, so there must be $b \in B \cap (U \cup V \cup W)$. By repeating the argument above, we can show that $rot^{-1} \in \langle rot_{A,B} \rangle$, so we are done.

(iv) From (ii) we have that RHS $\subseteq LHS$. To prove the other direction, we need to identify relations that $\langle rot \rangle$ preserves. These relations correspond to the orbits when you let $\gcl(rot)$ act on $(D,E)$. We describe the orbits diagrammatically:

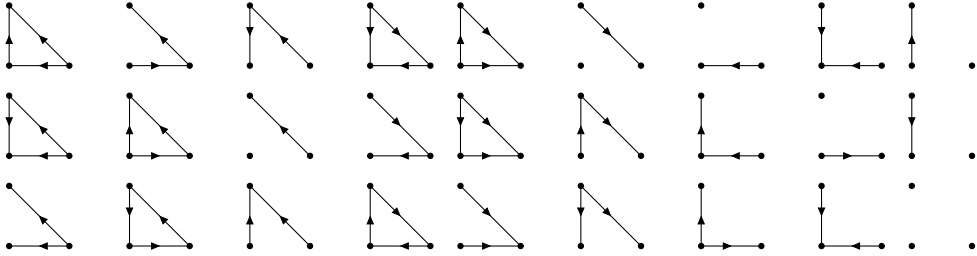
\begin{center}
\begin{tikzpicture}[line cap=round,line join=round,>=triangle 45,x=0.2cm,y=0.2cm]
\begin{scope}[decoration={markings,mark=at position 0.5 with {\arrow[scale=0.5]{>}}}] 
\clip(2.,-9.) rectangle (69.,9.);
\draw[postaction={decorate}]  (8.,4.) -- (4.,8.);
\draw[postaction={decorate}]  (4.,4.) -- (4.,8.);
\draw[postaction={decorate}]  (8.,4.) -- (4.,4.);
\draw[postaction={decorate}]  (12.,4.) -- (16.,4.);
\draw[postaction={decorate}]  (16.,4.) -- (12.,8.);
\draw[postaction={decorate}]  (20.,8.) -- (20.,4.);
\draw[postaction={decorate}]  (24.,4.) -- (20.,8.);
\draw[postaction={decorate}]  (28.,8.) -- (28.,4.);
\draw[postaction={decorate}]  (28.,8.) -- (32.,4.);
\draw[postaction={decorate}]  (32.,4.) -- (28.,4.);
\draw[postaction={decorate}]  (34.,4.) -- (34.,8.);
\draw[postaction={decorate}]  (34.,4.) -- (38.,4.);
\draw[postaction={decorate}]  (34.,8.) -- (38.,4.);
\draw[postaction={decorate}]  (42.,8.) -- (46.,4.);
\draw[postaction={decorate}]  (54.,4.) -- (50.,4.);
\draw[postaction={decorate}]  (58.,8.) -- (58.,4.);
\draw[postaction={decorate}]  (62.,4.) -- (58.,4.);
\draw[postaction={decorate}]  (64.,4.) -- (64.,8.);
\draw[postaction={decorate}]  (4.,2.) -- (4.,-2.);
\draw[postaction={decorate}]  (8.,-2.) -- (4.,-2.);
\draw[postaction={decorate}]  (8.,-2.) -- (4.,2.);
\draw[postaction={decorate}]  (12.,-2.) -- (12.,2.);
\draw[postaction={decorate}]  (12.,-2.) -- (16.,-2.);
\draw[postaction={decorate}]  (16.,-2.) -- (12.,2.);
\draw[postaction={decorate}]  (24.,-2.) -- (20.,2.);
\draw[postaction={decorate}]  (32.,-2.) -- (28.,-2.);
\draw[postaction={decorate}]  (28.,2.) -- (32.,-2.);
\draw[postaction={decorate}]  (34.,2.) -- (34.,-2.);
\draw[postaction={decorate}]  (34.,-2.) -- (38.,-2.);
\draw[postaction={decorate}]  (34.,2.) -- (38.,-2.);
\draw[postaction={decorate}]  (42.,-2.) -- (42.,2.);
\draw[postaction={decorate}]  (42.,2.) -- (46.,-2.);
\draw[postaction={decorate}]  (54.,-2.) -- (50.,-2.);
\draw[postaction={decorate}]  (50.,-2.) -- (50.,2.);
\draw[postaction={decorate}]  (58.,-2.) -- (62.,-2.);
\draw[postaction={decorate}]  (64.,2.) -- (64.,-2.);
\draw[postaction={decorate}]  (8.,-8.) -- (4.,-4.);
\draw[postaction={decorate}]  (8.,-8.) -- (4.,-8.);
\draw[postaction={decorate}]  (16.,-8.) -- (12.,-4.);
\draw[postaction={decorate}]  (12.,-4.) -- (12.,-8.);
\draw[postaction={decorate}]  (12.,-8.) -- (16.,-8.);
\draw[postaction={decorate}]  (20.,-8.) -- (20.,-4.);
\draw[postaction={decorate}]  (24.,-8.) -- (20.,-4.);
\draw[postaction={decorate}]  (28.,-8.) -- (28.,-4.);
\draw[postaction={decorate}]  (28.,-4.) -- (32.,-8.);
\draw[postaction={decorate}]  (32.,-8.) -- (28.,-8.);
\draw[postaction={decorate}]  (34.,-4.) -- (38.,-8.);
\draw[postaction={decorate}]  (34.,-8.) -- (38.,-8.);
\draw[postaction={decorate}]  (42.,-4.) -- (42.,-8.);
\draw[postaction={decorate}]  (42.,-4.) -- (46.,-8.);
\draw[postaction={decorate}]  (50.,-8.) -- (50.,-4.);
\draw[postaction={decorate}]  (50.,-8.) -- (54.,-8.);
\draw[postaction={decorate}]  (58.,-4.) -- (58.,-8.);
\draw[postaction={decorate}]  (62.,-8.) -- (58.,-8.);
\begin{scriptsize}
\draw [fill=black] (4.,-2.) circle (1.0pt);
\draw [fill=black] (4.,2.) circle (1.0pt);
\draw [fill=black] (8.,-2.) circle (1.0pt);
\draw [fill=black] (4.,-4.) circle (1.0pt);
\draw [fill=black] (4.,-8.) circle (1.0pt);
\draw [fill=black] (8.,-8.) circle (1.0pt);
\draw [fill=black] (4.,4.) circle (1.0pt);
\draw [fill=black] (8.,4.) circle (1.0pt);
\draw [fill=black] (4.,8.) circle (1.0pt);
\draw [fill=black] (12.,-2.) circle (1.0pt);
\draw [fill=black] (12.,2.) circle (1.0pt);
\draw [fill=black] (16.,-2.) circle (1.0pt);
\draw [fill=black] (12.,-4.) circle (1.0pt);
\draw [fill=black] (12.,-8.) circle (1.0pt);
\draw [fill=black] (16.,-8.) circle (1.0pt);
\draw [fill=black] (12.,4.) circle (1.0pt);
\draw [fill=black] (16.,4.) circle (1.0pt);
\draw [fill=black] (12.,8.) circle (1.0pt);
\draw [fill=black] (20.,-2.) circle (1.0pt);
\draw [fill=black] (20.,2.) circle (1.0pt);
\draw [fill=black] (24.,-2.) circle (1.0pt);
\draw [fill=black] (20.,-4.) circle (1.0pt);
\draw [fill=black] (20.,-8.) circle (1.0pt);
\draw [fill=black] (24.,-8.) circle (1.0pt);
\draw [fill=black] (20.,4.) circle (1.0pt);
\draw [fill=black] (24.,4.) circle (1.0pt);
\draw [fill=black] (20.,8.) circle (1.0pt);
\draw [fill=black] (28.,-2.) circle (1.0pt);
\draw [fill=black] (28.,2.) circle (1.0pt);
\draw [fill=black] (32.,-2.) circle (1.0pt);
\draw [fill=black] (28.,-4.) circle (1.0pt);
\draw [fill=black] (28.,-8.) circle (1.0pt);
\draw [fill=black] (32.,-8.) circle (1.0pt);
\draw [fill=black] (28.,4.) circle (1.0pt);
\draw [fill=black] (32.,4.) circle (1.0pt);
\draw [fill=black] (28.,8.) circle (1.0pt);
\draw [fill=black] (34.,-2.) circle (1.0pt);
\draw [fill=black] (34.,2.) circle (1.0pt);
\draw [fill=black] (38.,-2.) circle (1.0pt);
\draw [fill=black] (34.,-4.) circle (1.0pt);
\draw [fill=black] (34.,-8.) circle (1.0pt);
\draw [fill=black] (38.,-8.) circle (1.0pt);
\draw [fill=black] (34.,4.) circle (1.0pt);
\draw [fill=black] (38.,4.) circle (1.0pt);
\draw [fill=black] (34.,8.) circle (1.0pt);
\draw [fill=black] (42.,-2.) circle (1.0pt);
\draw [fill=black] (42.,2.) circle (1.0pt);
\draw [fill=black] (46.,-2.) circle (1.0pt);
\draw [fill=black] (42.,-4.) circle (1.0pt);
\draw [fill=black] (42.,-8.) circle (1.0pt);
\draw [fill=black] (46.,-8.) circle (1.0pt);
\draw [fill=black] (42.,4.) circle (1.0pt);
\draw [fill=black] (46.,4.) circle (1.0pt);
\draw [fill=black] (42.,8.) circle (1.0pt);
\draw [fill=black] (50.,-2.) circle (1.0pt);
\draw [fill=black] (50.,2.) circle (1.0pt);
\draw [fill=black] (54.,-2.) circle (1.0pt);
\draw [fill=black] (50.,-4.) circle (1.0pt);
\draw [fill=black] (50.,-8.) circle (1.0pt);
\draw [fill=black] (54.,-8.) circle (1.0pt);
\draw [fill=black] (50.,4.) circle (1.0pt);
\draw [fill=black] (54.,4.) circle (1.0pt);
\draw [fill=black] (50.,8.) circle (1.0pt);
\draw [fill=black] (58.,-2.) circle (1.0pt);
\draw [fill=black] (58.,2.) circle (1.0pt);
\draw [fill=black] (62.,-2.) circle (1.0pt);
\draw [fill=black] (58.,-4.) circle (1.0pt);
\draw [fill=black] (58.,-8.) circle (1.0pt);
\draw [fill=black] (62.,-8.) circle (1.0pt);
\draw [fill=black] (58.,4.) circle (1.0pt);
\draw [fill=black] (62.,4.) circle (1.0pt);
\draw [fill=black] (58.,8.) circle (1.0pt);
\draw [fill=black] (64.,-2.) circle (1.0pt);
\draw [fill=black] (64.,2.) circle (1.0pt);
\draw [fill=black] (68.,-2.) circle (1.0pt);
\draw [fill=black] (64.,-4.) circle (1.0pt);
\draw [fill=black] (64.,-8.) circle (1.0pt);
\draw [fill=black] (68.,-8.) circle (1.0pt);
\draw [fill=black] (64.,4.) circle (1.0pt);
\draw [fill=black] (68.,4.) circle (1.0pt);
\draw [fill=black] (64.,8.) circle (1.0pt);
\end{scriptsize}
\end{scope}
\end{tikzpicture}
\end{center}

This diagram contains all the possible digraphs you can have on a triple in $D$. Each row of the diagram represents one of the orbits and hence, one of the relations that $\langle rot \rangle$ preserves. Let $P_{rot,1}, P_{rot,2}$ and $P_{rot,3}$ be the relations for the top, middle and bottom rows respectively. One feature worth noting is that given any finite triple in $D$, if you change the relation between exactly one pair of its vertices, you change the orbit the triple is in. For example, given a triple with only non-edges (so it is in $P_{rot,3}$), changing exactly one non-edge into an edge results in the triple no longer being in $P_{rot,3}$.

Now let $f \in \langle rot \rangle$. We know that $f$ preserves $P_{rot,i}$, $i=1,2,3$. We want to find disjoint $A,B \subseteq D$ such that $f$ behaves like $rot_{A,B}$. We do this as follows. Pick any $a \in D$. Let $A=\{a\} \cup \{x \in D: E(a,x) \wedge E(f(a,x))$ or $E^*(a,x) \wedge E^*(f(a,x))$ or $N(a,x) \wedge N(f(a,x)) \}$.  Let $B=\{x \in D: E(a,x) \wedge E^*(f(a,x))$ or $E^*(a,x) \wedge N(f(a,x))$ or $N(a,x) \wedge E(f(a,x)) \}$.

We claim that $f$ behaves like $rot_{A,B}$. This amounts to case checking, which we leave to the reader. We provide one case as an example.

Case 1. We need to show that $f$ behaves like $id$ on $A$.  Suppose not, and let $a_1,a_2 \in A$ witness this fact.  Then we have $(a,a_1,a_2)$ such that $f$ only changes what happens between $a_1$ and $a_2$, contradicting that $f$ preserves $P_{rot,i}$.
\end{proof}

The relations introduced in this proof are important, so we give them a definition.

\begin{Def} For $i=1,2,3$, let $P_{rot,i}(x,y,z)$ be the relations defined in the proof of part (iv) of the lemma above.
\end{Def}


The descriptions of $\langle -, sw \rangle$ and $\langle -, rot \rangle$ are straightforward:

\begin{lem} \label{understandingjoinwith-}
\begin{enumerate}[(i)]
\item $\langle -,sw \rangle = \{f \in$ Sym$(D): f=g$ or $- \circ g$ for some $g \in \langle sw \rangle \}$.
\item $\langle -,rot \rangle = \{f \in$ Sym$(D): f=g$ or $- \circ g$ for some $g \in \langle rot \rangle \}$.
\end{enumerate}
\end{lem}
\begin{proof} 
(i)$\langle -,sw \rangle$ preserves the 6-ary relation $P_{sw,w} \defeq P_{sw}(\bar{x}) \leftrightarrow P_{sw}(\bar{y})$. Now let $f \in \langle -,sw \rangle$. If $f$ preserves $P_{sw}$, then by \autoref{understandingsw} $f \in \langle sw \rangle$. Now suppose that $f$ does not preserve $P_{sw}$. Since $f$ preserves $P_{sw,w}$, we have that $- \circ f$ preserves $P_{sw}$, so $- \circ f=g \in \langle sw \rangle$. Hence, $f=-^{-1}g$. We can replace $-^{-1}$ by $-$ because $-^{-1}=-\circ h$ for some $h \in$ Aut$(D,E)$.

(ii)$\langle -,rot \rangle$ preserves the 6-ary relation $P_{rot,w} \defeq (P_{rot,1}(\bar{x}) \wedge P_{rot,1}(\bar{y})) \vee (P_{rot,2}(\bar{x}) \wedge P_{rot,2}(\bar{y}))$. Now let $f \in \langle -,rot \rangle$. If $f$ preserves $P_{rot,1}$, then by \autoref{understandingrot} $f \in \langle rot \rangle$. Now suppose that $f$ does not preserve $P_{rot,1}$. Since $f$ preserves $P_{rot,w}$, we have that $- \circ f$ preserves $P_1$, so $- \circ f=g \in \langle rot \rangle$. Hence, $f=-^{-1}g$. We can replace $-^{-1}$ by $-$ because $-^{-1}=-\circ h$ for some $h \in$ Aut$(D,E)$.
\end{proof}

The next lemmas will give us conditions on a group $G$ to be equal to Sym$(D)$ or to contain Aut($\Gamma$).

\begin{lem} \label{equalsSymD} Let $G \leq$ Sym$(D)$ be a closed supergroup of Aut$(D,E)$.
\begin{enumerate}[(i)]
\item If $G$ is $n$-transitive for all $n \in \mathbb{N}$, then $G=$ Sym$(D)$.
\item If $G$ is $n$-homogeneous for all $n \in \mathbb{N}$, then $G=$ Sym$(D)$.
\item Suppose that whenever $A \subset D$ is finite and has edges, there exists $g \in G$ such that $g(A)$ has less edges than in $A$ (i.e. $|\{(x,y) \in A^2: E(g(x),g(y))\}| <  |\{(x,y) \in A^2: E(x,y)\}|)$. Then, $G=$ Sym$(D)$.
\item Suppose that there exists a finite $A \subset D$ and $g \in G$ such that $g$ behaves like $id$ on $D\backslash A$, $g$ behaves like $id$ between $A$ and $D\backslash A$, and, $g$ deletes at least one edge in $A$. Then, $G=$ Sym$(D)$.
\end{enumerate}
\end{lem}

Remark: $G$ is $n$-transitive if for all pairs of tuples $\bar{x}, \bar{y} \in D^n$, there exists $g \in G$ such that $g(\bar{x})=\bar{y}$.  $G$ is $n$-homogeneous if for all subsets $A,B \subset D$ of size $n$, there exists $g \in G$ such that $g(A)=B$.

\begin{proof}
(i) Let $f \in$ Sym$(D)$. We want to show that $f \in G$. Since $G$ is closed, it suffices to show that for all finite tuples $\bar{a} \in D$, there exists $g \in G$ such that $g$ maps $\bar{a}$ to $\bar{f(a)}$. But $G$ is $n$-transitive for all $n$, so we can always find an appropriate $g$, so we are done.

(ii) We will show that $G$ is $n$-transitive for all $n$. Let $\bar{a},\bar{b}$ be tuples of length $n$ in $D$. Let $f \in G$ be such that $f(\bar{a})$ is empty; this is possible as $G$ is $n$-homogeneous. Similarly, let $g \in G$ be such that $g(\bar{b})$ is empty. Now consider the map $h: f(a_i) \mapsto g(b_i)$. This is an isomorphism of digraphs so can be extended to an automorphism $h'$ of Aut$(D,E)$, by homogeneity. But now $g^{-1}h'f \in G$ maps the tuple $\bar{a}$ to the tuple $\bar{b}$, as required.

(iii) We will show that $G$ is $n$-homogeneous for all $n$. It suffices to show that for all finite $A \subset D$, we can map $A$ to the empty digraph. We prove this by induction on the number of edges $k$ in $A$. The base case $k=0$ is trivial. Now let $A$ have $k$ edges. By assumption, there is $f \in G$ such that $f(A)$ has $k'<k$ edges. By the inductive hypothesis, there is $g \in G$ such that $g(f(A))$ is the empty digraph, so we are done.

(iv) Let $A$ and $g$ be as in the lemma. We will show that for all finite $B \subset D$, if $B$ contains edges then there is $f \in G$ such that $f(B)$ has less edges than $B$ - this suffices by (iii). So let $B \subset D$ be finite. Let $bb'$ be an edge in $B$, and let $aa' \in A$ be an edge that is deleted by $g$. Let $h$ be an automorphism mapping $bb'$ to $aa'$. Then $gh \in G$ and $gh(B)$ contains less edges than in $B$, as required.
\end{proof}

Before we describe conditions for $G$ to contain Aut$(\Gamma)$, we first establish a fact we have mentioned earlier, which is that $\Gamma$ is indeed the random graph.

\begin{lem} $\Gamma$ is isomorphic to the random graph.
\end{lem}
\begin{proof}
Recall that we defined $\Gamma$ to be $(D,E_{\Gamma})$, where $E_{\Gamma}(x,y) \defeq E(x,y) \vee E(y,x)$. To show $\Gamma$ is isomorphic to the random graph, it suffices to show it satisfies the extension property of the random graph. So let $U,W \subset D$ be finite disjoint - we need to find $x \in D \backslash (U \cup W)$ such that $E_{\Gamma}(x,u)$ for all $u \in U$ and $N(x,w)$ for all $w \in W$. Apply the extension property of the digraph (\autoref{genericdigraph}) to $U, \emptyset, W$ to find an appropriate $x$.
\end{proof}

\begin{lem} \label{containsgamma} Let $G \leq$ Sym$(D)$ be a closed supergroup of Aut$(D,E)$.
\begin{enumerate}[(i)]
\item Suppose that whenever $a_1,\ldots,a_n, b_1, \ldots, b_n \in D$ satisfy $N(a_i,a_j) \leftrightarrow N(b_i,b_j)$ for all $i,j$, there exists $g \in G$ such that $g(\bar{a})=\bar{b}$. Then $G \geq$ Aut$(\Gamma)$.
\item Suppose that for all $A=\{a_1,\ldots,a_n\} \subset D$, there exists $g \in G$ such that for all edges $a_ia_j$ in $A$, $E(g(a_i),g(a_j))$ iff $i<j$. (Intuitively, such a $g$ is switching the edges so they all point in the same direction.) Then, $G \geq$ Aut$(\Gamma)$.
\item Suppose that there exists a finite $A \subset D$ and $g \in G$ such that $g$ behaves like $id$ on $D \backslash A$, $g$ behaves like $id$ between $A$ and $D \backslash A$, and, $g$ switches the direction of (at least) one edge in $A$. Then, $G \geq$ Aut$(\Gamma)$.
\end{enumerate}
\end{lem}

\begin{proof}
(i) Let $f \in$ Aut$(\Gamma)$ and let $\bar{a} \in D$ be a finite tuple. We need to find $g \in G$ such that $g(\bar{a})=f(\bar{a})$. Since $f \in$ Aut$(\Gamma)$, we have that $N(a_i,a_j) \leftrightarrow N(f(a_i),f(a_j))$ for all $i,j$. Hence, by the assumptions given in the lemma, there exists an appropriate $g \in G$.

(ii) Let $\bar{a}$ and $\bar{b} \in D$ satisfy $N(a_i,a_j) \leftrightarrow N(b_i,b_j)$ - we will show that there is $f \in G$ s.t $f(\bar{a})=\bar{b}$. Let $g_1 \in G$ be a function such that for all edges $a_ia_j \in A$, $E(g(a_i),g(a_j))$ iff $i<j$; such a function exists by assumption. Let $g_2 \in G$ be the corresponding function for $\bar{b}$. By construction, $g_1(\bar{a})$ and $g_2(\bar{a})$ are isomorphic so there is an automorphism $h \in$ Aut$(D,E)$ mapping $g_1(\bar{a})$ to $g_2(\bar{a})$. But then $g_2^{-1}hg_1(\bar{a})=\bar{b}$. Hence we are done by part (i).

(iii) Let $A$ and $g$ be as stated in the lemma, and let $aa' \in A$ be an edge whose direction is switched by $g$.

Claim: Let $\bar{b} \in D$ be finite and let $bb'$ be any edge in $\bar{b}$. Then there exists $f \in G$ such that $f$ switches the direction of $bb'$ and behaves like $id$ everywhere else on $\bar{b}$. This is easy: By homogeneity, there exists $h \in$ Aut$(D,E)$ such that $h(bb')=aa'$ and $h(\bar{b}) \cap A = \{a,a'\}$. Then $f=gh$ is the function we want.

Now suppose we have two tuples $\bar{b},\bar{c}$ as in the statement of (i); we want to find a function in $G$ mapping one to the other. We do this by repeatedly using the above claim to switch the edges in $\bar{b}$ until they are all aligned with the edges in $\bar{c}$. 
\end{proof}


\section{$\mathcal{L}$ is a sublattice of the reducts of $(D,E)$}
Before we begin please note a convention that we will use for the remainder of the article. There will be proofs where we want to show that we can map a digraph $A$ to a related digraph $B$. Often, the function will be the composition of a sequence of functions $f_1,f_2,\ldots$, where the definition of each one will depend on those defined earlier. For example, suppose we have defined $f_1$ and $f_2$, and $f_3$ is going to be a switching function. The convention is that we will say `Let $f_3$ be $sw_{A'}$' (where $A'$ will be a particular subset of $A$), in place of the strictly correct phrase `Let $f$ be $sw_{f_2f_1(A')}$'.

This may seem odd, but it has benefits. First, the proofs will be easier to follow and will better match the underlying intuition behind the argument. Second, with this convention in place, we often avoid needing to name the functions: We can now use phrases like `First switch about the subset $A_1$, then apply $rot$ about the point $a$', whereas without the convention we would have to say `...then apply $rot$ about the point which is the current image of $a$'.

\begin{lem} \label{sublattice}
\begin{enumerate}[(i)]
\item $\langle - \rangle, \langle sw \rangle$ and $\langle rot \rangle$ are proper reducts of Aut$(D,E)$.
\item $\langle - \rangle, \langle sw \rangle$ and $\langle rot \rangle$ are not reducts of each other.
\item $\langle -, sw \rangle$ is a proper reduct of $\langle - \rangle$ and $\langle sw \rangle$, and is not equal to Sym$(D)$.
\item $\Gamma$ is a proper reduct of $\langle -, sw \rangle$
\item $\langle -,rot \rangle$ is a proper reduct of $\langle - \rangle$ and $\langle rot \rangle$, and is not equal to Sym$(D)$.
\item The join of $\langle rot \rangle$ and $\langle sw \rangle$ is Sym$(D)$.
\item The meet of $\langle sw \rangle$ and $\langle - \rangle$ is Aut$(D)$.
\item The meet of $\langle rot \rangle$ and $\langle sw_{\Gamma}, -_{\Gamma} \rangle$ is Aut$(D)$.
\item The meet of $\langle -,rot \rangle$ and $\langle sw_{\Gamma}, -_{\Gamma} \rangle$ is $\langle - \rangle$.
\end{enumerate}
\end{lem}
\begin{proof} (i) This is immediate from the definition of the $\langle \cdot \rangle$.

(ii) We need to identify for each reduct a relation that it preserves but which the other two do not preserve. For $\langle - \rangle$ the relation is $E_w$, for $\langle sw \rangle$ the relation it preserves is $P_{sw}$ and for $\langle rot \rangle$ we have $P_{rot, 1}$.

(iii) By (ii), $\langle -,sw \rangle$ is a proper reduct of $\langle - \rangle$ and $\langle sw \rangle$. It preserves $P_{sw,w}$, so it is not equal to Sym$(D)$.

(iv) Both $-$ and $sw$ preserve $N(x,y)$, so $\langle -, sw \rangle \subseteq$ Aut$(D,N)=\Gamma$.  $\Gamma$ is a proper reduct because $\langle -, sw \rangle$ preserves $P_{sw,w}$ but $\Gamma$ does not.

(v) By (ii), $\langle -,rot \rangle$ is a proper reduct of $\langle - \rangle$ and $\langle rot \rangle$. It preserves $P_{rot,w}$, so it is not equal to Sym$(D)$.

(vi) By \autoref{equalsSymD} (iii), it suffices to show that for all finite $A \subset D$ that has at least one edge, we can find $g \in \langle sw, rot \rangle$ such that $g(A)$ has less edges than in $A$.

Let $a \in A$ be a point adjacent to at least one edge. Let $A_1 = \{a' \in A: E(a,a')\}, A_2=\{a' \in A: E(a',a)\}$ and $A_3=\{a' \in A: N(a,a')\}$. First, switch about the subset $A_1$ - the result is that now all the edges adjacent to $a$ are edges going into $a$. Now apply $rot_a^2$: the edges between $a$ and $A_1 \cup A_2$ become outward edges, and the non-edges between $a$ and $A_3$ become inward edges. Now apply $sw_{A_1 \cup A_2}$: the outward edges from $a$ to $A_1 \cup A_2$ now become inward edges. Therefore, between $a$ and $A \backslash \{a\}$ we now only have inward edges. Applying $rot_a$ for the last time results in all these edges becoming non-edges. By noting that at every step, the number of edges within $A \backslash \{a\}$ remains the same, we have shown that we can reduce the number of edges in $A$ using functions in $\langle sw,rot \rangle$, which is what was required.

(vii) Let $f \in \langle - \rangle \cap \langle sw \rangle$. By \autoref{understandingsw}, $f$ behaves like $sw_A$ for some $A \subseteq D$. $A$ or $A^c$ must contain an edge. Hence, there exists an edge whose direction $f$ does not switch. In particular, $f$ does not behave like $-$. By \autoref{understanding-}, we conclude that $f$ has to be an automorphism of Aut$(D,E)$, as required.

(viii) We first establish some notation. We say $f: D \to D$ graph-behaves like $g: D \to D$ if for all $\bar{a} \in D$, $f(\bar{a})$ is isomorphic to $g(\bar{a})$ as \emph{undirected} graphs. We abbreviate `graph-behaves' by `g-behaves'. Let $A \subseteq D$. We say $f: D \to D$ g-behaves like $sw_{\Gamma,A}$ if $f$ g-behaves like $id$ on $A$ and on $A^c$ and if $f$ swaps edges and non-edges between $A$ and $A^c$. By folklore (or by duplicating the arguments in Section 2), $\langle sw_{\Gamma} \rangle = \{f \in$ Sym$(D): f$ g-behaves like $sw_{\Gamma,A}$ for some $A \subseteq D\}$, and $\langle -_{\Gamma}, sw_{\Gamma} \rangle = \{f \in$ Sym$(D): \exists g \in \langle sw_{\Gamma} \rangle$ such that $f=g$ or $f=-_{\Gamma}\circ g \}$.

Let $f \in \langle rot \rangle \cap \langle sw_{\Gamma}, -_{\Gamma} \rangle$. By \autoref{understandingrot}, there exists disjoint $A,B \subseteq D$ such that $f$ behaves like $rot_{A,B}$; let $C=(A \cup B)^c$. We split into two cases.

Case 1. $f \in \langle sw_{\Gamma} \rangle$, so $f$ g-behaves like $sw_{\Gamma,U}$ for some $U \subseteq D$; let $V=U^c$. To show that $f \in$ Aut$(D,E)$ it suffices to show that two of $A,B$ and $C$ must be empty. Suppose without loss that $A$ is non-empty, so we want to show that $B$ and $C$ are empty.

Since $f$ behaves like $id$ on $A$, $A$ must be a subset of $U$ or a subset of $V$. Without loss, suppose $A \subseteq U$. Similarly, $B$ and $C$ must each be a subset of $U$ or $V$. Furthermore, if $B$ is non-empty it cannot be a subset of $U$; this is because $f$ preserves non-edges in $U$ but $f$ does not preserve non-edges between $A$ and $B$. Similarly, if $C$ is non-empty, then $C \subseteq V$. So if both $B$ and $C$ are non-empty, then they must both be subsets of $V$, which is not possible by the same reasoning. Hence, one of $B$ or $C$ must be empty - without loss we may assume that $C$ is empty.

Now we have that $B=A^c$ is non-empty, $A \subseteq U$ and $B \subseteq V$. Hence, $A=U$ and $B=V$. By homogeneity of $D$, there must be an outward edge from $A$ to $B$. But now we get a contradiction: $f$ behaving like $rot_{A,B}$ implies that this edge is mapped to an edge, whereas $f$ g-behaving like $sw_{\Gamma,U}$ implies that $f$ maps this edge to a non-edge.  Thus, $B$ must also be non-empty, as required.

Case 2. $f=-_{\Gamma}\circ g$ for some $g \in \langle sw_{\Gamma} \rangle$. Let $U \subseteq D$ be such that $g$ g-behaves like $sw_{\Gamma,U}$. Now, for any subset $X$ of $D$ of size at least three, $f$ cannot act like the $id$ on $X$. This is because either $|X \cap U| \geq 2$ or $|X \cap U^c| \geq 2$ and we know that $f$ g-behaves like $-_{\Gamma}$ on $U$ and on $U^c$. However, we also know that $f$ behaves like the $id$ on $A,B$ and $C$, and at least one of them has size at least three. Thus, we have a contradiction.

(ix) Let $f \in \langle -,rot \rangle \cap \langle sw_{\Gamma}, -_{\Gamma} \rangle$. By \autoref{understandingjoinwith-}, there exists $g \in \langle rot \rangle$ such that $f=g$ or $f=-\circ g$. Since $f \in \langle - \rangle \Leftrightarrow - \circ f \in \langle - \rangle$, without loss we may assume that $f=g$, i.e. that $f \in \langle rot \rangle$. By (viii), it follows that $f \in$ Aut$(D,E)$, so we are done.
\end{proof}


\section{$\mathcal{L}$ contains all the reducts}
The task of showing that $\mathcal{L}$ contains all the reducts is split up into these lemmas:

\begin{lem}\label{threeregions} Let $G$ be a reduct of Aut$(D,E)$. Then either $G$ contains Aut$(\Gamma)$, is contained in Aut$(\Gamma)$,  or contains $\langle rot \rangle.$
\end{lem}

\begin{lem}\label{region1} Let $G$ be a reduct of Aut$(D,E)$ that contains Aut$(\Gamma)$. Then $G=\Gamma, \langle sw_{\Gamma} \rangle,$ $\langle -_{\Gamma} \rangle, \langle sw_{\Gamma}, -_{\Gamma} \rangle$ or Sym$(D)$.
\end{lem}

\begin{lem}\label{region2} Let $G$ be a reduct of Aut$(D,E)$ that is contained in Aut($\Gamma)$. Then $G=$Aut$(D,E)$, $\langle sw \rangle$, $\langle - \rangle, \langle sw, - \rangle$ or Aut$(\Gamma)$.
\end{lem}

\begin{lem} \label{region3} Let $G$ be a reduct of Aut$(D,E)$ that contains $\langle rot \rangle$. Then $G=\langle rot \rangle, \langle rot,- \rangle$ or Sym$(D)$.
\end{lem}

The main tool that will be used to prove these lemmas will be that of canonical functions, as developed by Bodirsky and Pinsker in \cite{bp11} and \cite{bpt13}. However, before delving into the use of canonical functions, the next subsection describes the details that are obtained by other means.

\subsection{Using the classification of the reducts of the random graph and of the random tournament}
Knowing the reducts of the random graph is evidently necessary for this result, but it is also helpful to know the reducts of the random tournament. We begin by stating these two classifications.

\textbf{Notation.}
\begin{enumerate}[(i)]
\item We let $\mathcal{T}=(T,E_T)$ denote the random tournament. This can be defined as the countable homogeneous tournament which embeds all finite tournaments.
\item Let $-_{\mathcal{T}}$ denote a function which switches the direction of all edges in the random tournament.
\item Let $sw_{\mathcal{T}}$ denote a function which switches the direction of only those edges that are adjacent to a particular fixed vertex.
\end{enumerate}

\begin{thm}\label{gammaandt}
\begin{enumerate}[(i)]
\item (Thomas \cite{tho91}.) The reducts of the random graph are: $\Gamma, \langle sw_{\Gamma} \rangle,$ $\langle -_{\Gamma} \rangle, \langle sw_{\Gamma}, -_{\Gamma} \rangle$ and the full symmetric group.
\item (Bennett, \cite{ben97}.) The reducts of the random tournament are: Aut$(T,E_T),\langle sw_{\mathcal{T}} \rangle, \langle -_{\mathcal{T}} \rangle,$ $\langle sw_{\mathcal{T}}, -_{\mathcal{T}} \rangle$ and the full symmetric group Sym($T$).
\end{enumerate}
\end{thm}

We immediately get:

\begin{proof}[Proof of \autoref{region1}] This is exactly the statement of \autoref{gammaandt} (i).
\end{proof}

Knowing the reducts of the random tournament contributes to the proof of \autoref{region2}, via the following construction:

\begin{Def}
Let $G$ be a reduct of $(D,E)$. We let $T(G)=\{ f \in$ Sym$(T):$ for all finite tuples $\bar{a} \in T,$ there exist $g \in G$ and a tuple $\bar{b} \in D$ such that $\bar{a} \cong \bar{b}$ and $f(\bar{a}) \cong g(\bar{b})\}$.
\end{Def}

In words, $T(G)$ contains those functions whose behaviour on finite sets can be replicated by functions in $G$. The intuition is that $T(G)$ tells us what $G$ can do to tournaments. The idea behind this concept is as follows: We show that $T(G)$ must be a reduct of $\mathcal{T}$, so by \autoref{gammaandt} $T(G)$ has five different possibilities. Now if we assume that $G$ fixes non-edges, $G$ can only change the direction of edges. From this, one might suspect that $G$ is determined by how it behaves on tournaments, i.e., that $G$ is determined by $T(G)$.

\begin{lem} \label{region2a} Let $G$ be a reduct of $(D,E)$. Then $T(G)$ is a reduct of $\mathcal{T}$.
\end{lem}
\begin{proof}
We need to show that $T(G)$ is a closed supergroup of Aut$\mathcal(T)$. This is an easy exercise in unravelling definitions. We demonstrate by showing that $T(G)$ is closed under composition, and leave the remaining conditions to the reader.

Let $f,f' \in T(G)$. We want to show that $f'f \in T(G)$, so let $\bar{a} \in T$ be a finite tuple. Since $f \in T(G)$ we can find $g \in G$ and $\bar{b} \in D$ such that $\bar{a} \cong \bar{b}$ and $f(\bar{a}) \cong g(\bar{b})$. Since $f' \in T(G)$, we can find $g' \in G$ such that $f' (f(\bar{a})) \cong g' (g(\bar{b}))$. Then $g'g$ and $\bar{b}$ satisfy $\bar{a} \cong \bar{b}$ and $f'f(\bar{a}) \cong g'g(\bar{b})$, as required.
\end{proof}

\begin{lem} \label{region2b} Let $G$ be a reduct of $(D,E)$ contained in Aut$(\Gamma)$. Then:
\begin{enumerate}[(i)]
\item $G=$Aut$(D,E) \Leftrightarrow T(G)=$Aut$(T,E_T)$.
\item $G=\langle sw \rangle \Leftrightarrow T(G) = \langle sw_{\mathcal{T}} \rangle$.
\item $G=\langle - \rangle \Leftrightarrow T(G) = \langle -_{\mathcal{T}} \rangle $.
\item $G=\langle sw,- \rangle \Leftrightarrow T(G) = \langle sw_{\mathcal{T}}, -_{\mathcal{T}} \rangle$.
\end{enumerate}
\end{lem}

\begin{proof}
The following claims are used in all four parts of the lemma.

textbf{Claim 1.} $T(G)=$ Aut$(T,E_T), \langle sw_{\mathcal{T}} \rangle, \langle -_{\mathcal{T}} \rangle,$ $\langle sw_{\mathcal{T}}, -_{\mathcal{T}} \rangle$ or Sym$(T)$.

\emph{Proof of Claim 1.} This follows immediately from \autoref{region2a} and \autoref{gammaandt}.

\textbf{Claim 2.} Let $g \in G$ and let $\bar{b} \in D$ be a tournament. Then there exist $f \in T(G)$ and $\bar{a} \in T$ s.t $\bar{a} \cong \bar{b}$ and $f(\bar{a}) \cong g(\bar{b})$.

\emph{Proof of Claim 2.} Let $T_1 \subset D$ satisfy: \begin{itemize}
\item $\bar{b} \in T_1$
\item $(T_1,E|_{T_1})$ is isomorphic to the random tournament.
\item $T_1$ is a maximal tournament in $D$, i.e. for all $x \in D\backslash T_1$, there exists $y \in T_1$ such that $N(x,y)$.
\end{itemize}

We sketch how one can show such a $T_1$ exists. Start with $(T,E_T)$, and let $D'=T \cup \{x_1,x_2,x_3,\ldots \}$. We want to define an edge relation on $D'$ so that it extends $E_T$, so that it satisfies the digraph extension property (so by \autoref{genericdigraph} we get the generic digraph), and so that $T$ is a maximal tournament in $D'$. The trickiest condition is ensuring the digraph extension property is satisfied: to deal with this, you enumerate all the pairwise disjoint triples $(U,V,W) \subset D'$, and then you define edge relations so that $x_i$ witnesses the extension property for the $i$th triple. Any edges which are not determined by this process are chosen to be non-edges - this ensures $T$ is a maximal tournament in $D'$.

By composing with an element of Aut$(D,E)$ if necessary, we can assume that $g(\bar{b}) \in T_1$.  Hence, and because elements of Aut$(\Gamma)$ map maximal tournaments to maximal tournaments, $g(T_1)=T_1$. 

Now, let $\theta: T \to T_1$ witness the fact that $T_1$ is isomorphic to the random tournament. Now let $f=\theta^{-1}g \theta$. It is easy to see that $f$ satisfies the requirements of the claim.

(i) ``$\Rightarrow$''. We prove the contrapositive, so suppose $T(G)$ does not equal Aut$(T,E_T)$. Then there exists $f \in T(G)$ which swapS the direction of some edge in $T$. By definition of $T(G)$, that means there is $g \in G$ which swaps the direction of some edge in $D$, which implies that $G \neq$ Aut$(D,E)$.

``$\Leftarrow$''. Suppose $G \neq$ Aut$(D,E)$. Hence, there exists $g \in G$ and an edge $b_1b_2 \in D$ such that $g$ switches the direction of that edge. Hence, by Claim 2, there exists $f \in T(G)$ which switches the direction of an edge, which implies that $T(G) \neq$ Aut$(T,E_T)$.

(ii) ``$\Rightarrow$''. By Claim 1, we have five options for $T(\langle sw \rangle)$. By (i), it cannot be Aut$(T,E_T)$. Suppose $T(G)$ contains $\langle -_T \rangle$. Then there exists $f \in T(G)$ and a triangle in $T$ such that $f$ swaps the direction of all three edges of the triangle. This implies that there is $g \in G$ which swaps the direction of all three edges of a triangle in $D$. But no such function exists in $\langle sw \rangle$, so if $T(G) \geq \langle -_T \rangle$, then $G \neq \langle sw \rangle$.  Hence, we have that $T(\langle sw \rangle) = \langle sw_T \rangle.$

``$\Leftarrow$''. Suppose $T(G)=\langle sw_T \rangle$. By Claim 2, if $G$ does not preserve $P_{sw}$, then this can be witnessed in $T(G)$ also. Since $\langle sw_T \rangle$ does preserve $P_{sw}$, we get that $G$ preserves $P_{sw}$. By \autoref{understandingsw}, we get that $G=$ Aut$(D,E)$ or $\langle sw \rangle$. But it cannot be the former option, so $G=\langle sw \rangle$.

(iii) Same arguments as for part (ii).

(iv) ``$\Rightarrow$''. This is proved similarly to previous cases.

``$\Leftarrow$''. Suppose $T(G)=\langle sw_T,-_T \rangle$. By Claim 2, we get that $G$ preserves $P_{sw,w}$, which implies that $G \leq \langle sw,- \rangle$. In $\langle sw_T,-_T \rangle$, there is a function that does not preserve $sw$. Hence, there is a function $g \in G$ which does not preserve $sw$. Hence, by \autoref{understandingjoinwith-}, $g=- \circ g'$ where $g' \in sw$. Then $g^2$ will be in $\langle sw \rangle \backslash$Aut$(D,E)$. Hence, by \autoref{understandingsw}, $G \geq \langle sw \rangle$. By composing $g$ with an appropriate element of $\langle sw \rangle$, we get that $- \in G$. Hence, we have that $G \geq \langle sw,- \rangle$.  Thus, $G=\langle sw,- \rangle$, as required.
\end{proof}

This lemma almost completes the proof of \autoref{region1}. What is left to prove is that if $\langle -,sw \rangle < G \leq$ Aut$(\Gamma)$, then $G=$ Aut$(\Gamma)$. We believe that this can be proved directly (without the need of canonical functions), but the combinatorics involved were just out of our reach.


\subsection{Canonical functions}
\begin{Def} Let $\mathcal{M}, \mathcal{N}$ be any structures. Let $f: M \to N$ be any function between the domains of the structures.
\begin{enumerate}[(i)]
\item The \emph{behaviour} of $f$ is the relation $\{(p,q) \in S(\mathcal{M})\times S(\mathcal{N}): \exists \bar{a} \in M, \bar{b} \in N$ such that tp$(\bar{a})=p$, tp$(\bar{b})=q$ and $f(\bar{a})=\bar{b} \}$.
\item If the behaviour of $f$ is a function $S(M) \to S(N)$, then we say $f$ is \emph{canonical}. Rephrased, we say $f$ is canonical if for all $\bar{a}, \bar{a}' \in M$, tp$(\bar{a})=$ tp$(\bar{a}') \Rightarrow$ tp$(f(\bar{a}))=$ tp$(f(\bar{a}'))$.
\item If $f$ is canonical, we use the same symbol $f$ to denote its behaviour.
\end{enumerate}
\end{Def}

\textbf{Examples.} \begin{enumerate}
\item Any $f \in$ Aut$(D,E)$ is a canonical function, and for all types $p$, $f(p)=p$.
\item $-$ is canonical.
\item $sw_a$ is not canonical: Let $b,b'$ be vertices such that we have $E(a,b)$ and $E(b',b)$. Then, tp$(a,b)=$ tp$(b',b)$, but tp$(sw(a,b)) \neq$ tp$(sw(b',b))$. Similarly, $rot_a$ is not canonical.
\item $sw_a$ and $rot_a$ \emph{are} canonical when we regard them as functions from $(D,E,a) \to (D,E)$.
\item Let $f,g$ be canonical functions. Then $f$ behaves like $g$ (in the sense of \autoref{behaves})  if and only if $f$ and $g$ have the same behaviour (in the sense of the definition above).  Note that this is not necessarily true if the functions are not canonical.
\end{enumerate}

The benefit of canonical functions is that they are particularly well-behaved and can be easily manipulated and analysed. The next theorem will be treated as a `black-box' for this article - a proof can be found in \cite{bpt13}.  In order to state the theorem, we need to give a couple of definitions.

\begin{Def}
Let $F \subseteq D^D$. We let $\tmcl(F)$\footnote{where $\tmcl$ stands for `topological monoid closure'} denote the smallest closed monoid in $M$ containing $F$. We may abuse notation and write $\tmcl(F)$ for $\tmcl($Aut$(D,E) \cup F)$.
\end{Def}

\begin{Def}
We let $(D,E,<)$ denote the countable (linearly) ordered homogeneous digraph that embeds all finite ordered digraphs. 
\end{Def}

The theorem that follows is an application of the theorem in \cite{bpt13} to the structure $(D,E,<)$. In order for this to be valid, we need to know that $(D,E,<)$ is a Ramsey structure. The definition of a Ramsey structure can be found in \cite{bpt13}. The fact that $(D,E,<)$ is Ramsey follows from the main theorem of \cite{nr77}.

\begin{thm}\label{blackbox} Let $f \in$ Sym$(D)$ and $c_1,\ldots, c_n \in D$ be any elements. Then there exists a function $g:D \to D$ such that \begin{enumerate}[(i)]
\item $g \in \tmcl($Aut$(D,E) \cup \{f\})$.
\item $g(c_i)=f(c_i)$ for $i=1,\ldots n$.
\item When regarded as a function from $(D,E,<,c_1,...c_n)$ to $(D,E)$, $g$ is a canonical function.
\end{enumerate}
\end{thm}

How is this theorem used? We illustrate by sketching how we will complete the proof of \autoref{region2}: $G$ is a closed group such that $\langle -,sw \rangle < G \leq$ Aut$(\Gamma)$. Thus, $G$ does not preserve $P_{sw,w}$; we let $f \in G$ and $c_1,\ldots, c_6 \in D$ witness this fact.  We now use \autoref{blackbox} to obtain the canonical $g$ as in the theorem. We then examine the possibilities for $g$'s behaviour, which boils down to some finite combinatorics. Using \autoref{containsgamma} we show that in all the possible behaviours, $G$ must contain Aut$(\Gamma)$.

Implicit in this argument is the fact that we care only about the behaviour of the canonical function. Though this is not immediate, it will certainly become clear as we work with these functions. Intuitively, the idea is that two different canonical functions $f,f'$ with the same behaviour provide the same information about $G$.

This means that when we analyse the canonical functions, it suffices to analyse the possible behaviours of canonical functions. This task in turn is greatly simplified by the following:

\textbf{Important Observation.} The behaviour of a canonical function $f: (D,E,<,c_1,\ldots, c_n) \to (D,E)$ is determined by the restriction of the behaviour to 2-types. This follows from two facts. The first is that $(D,E,<,c_1,\ldots,c_n)$ has quantifier elimination (see \cite{hod97}). The second is that the arity of the named relations is $\leq 2$. These two facts imply that the type of an $n$-tuple $(a_1,a_2,\ldots,a_n)$ is determined by the set of 2-types $\{$tp$(a_i,a_j): 1\leq i < j \leq n\}$; the observation follows easily from this.


\subsubsection{Canonical functions from $(D,E,<)$}
We start our analysis with the simplest situation, which is when no constants are added. As per the discussion above, it suffices to analyse the possible behaviours restricted to 2-types. To do this, we first need to describe what the possible 2-types of $(D,E,<)$ and $(D,E)$ are. 

\textbf{Notation.} Let $\phi_1(x,y),\ldots,\phi_n(x,y)$ be formulas. We let $p_{\phi_1,\ldots,\phi_n}(x,y)$ denote the (partial) type determined by the formula $\phi_1(x,y) \wedge \ldots \wedge \phi_n(x,y)$.

For example, let $a,b \in (D,E,<)$ be such that $a<b$ and $E(a,b)$. Then $p_{<,E}(x,y)=$ tp$(a,b)$. We will often omit the free variables $x$ and $y$ and write, for example, $p_{<,E}$.

With this notation in place, it is easy to state what the 2-types of $(D,E,<)$ and $(D,E)$ are.

\begin{itemize}
\item There are three 2-types in $(D,E)$: $p_{E},p_{E^*}$ and $p_{N}$.
\item There are six 2-types in $(D,E,<)$: $p_{<,E},p_{<,E^*}, p_{<,N}$, $p_{>,E},p_{>,E^*}$ and $p_{>,N}$.
\end{itemize}

Now, what are the possible behaviours? For each 2-type in $(D,E,<)$, we must choose which 2-type in $(D,E)$ it gets mapped to. This choice is not free: the image of a type $p(x,y)$, say, determines the image of the corresponding type $p^*(x,y) \defeq p(y,x)$. This is the only restriction - it is easy to show that all functions $\{p_{<,E},p_{<,E^*}, p_{<,N}\} \to \{p_E,p_{E^*},p_N\}$ can be realised as the behaviour of some canonical function $f:(D,E,<) \to (D,E)$. (You use the universality of $(D,E)$. Also, remember that we do not require $f$ to be bijective.)

The next lemma contains the analysis of these behaviours.

\begin{lem} \label{noconstants} Let $G$ be a closed supergroup of Aut$(D,E)$ and let $f \in \tmcl(G)$ be a canonical function from $(D,E,<)$ to $(D,E)$. Then (at least) one of the following is true:
\begin{itemize}
\item $f$ behaves like $id$.
\item $f$ behaves like $-$.
\item $G$ contains Aut($\Gamma$).
\end{itemize}
\end{lem}

\begin{proof} We split up the task according to the behaviour of $f$. For some of the cases, we use the following claim:

\textbf{Claim.~} When we consider $f$ as a function $(D,E,<) \to (D,E,<)$, we may assume that $f$ preserves the linear order.

\emph{Proof of Claim}. Let $f':D \to D$ be a function with the same behaviour as $f$ and which in addition preserves the linear order; we need to show that $f' \in \tmcl(G)$. Let $\bar{a} \in D$. By definition of $f'$, $f(\bar{a}) \cong f'(\bar{a})$ as unordered digraphs. By homogeneity of $(D,E)$, we can find $h \in$ Aut$(D,E)$ such that $h(f(\bar{a}))= f'(\bar{a})$. Since $f \in \tmcl(G)$, there is $g \in G$ such that $g(\bar{a})=f(\bar{a})$. So we have $hg \in G$ and $hg(\bar{a})=f'(\bar{a})$, as required.

\underline{Case 1}. $f(p_{<,N})=p_N$.

Case 1a. $f(p_{<,E})=p_{E}$ and $f(p_{<,E^*})=p_{E^*}$, in which case $f$ behaves like $id$.

Case 1b. $f(p_{<,E})=p_{E^*}$ and $f(p_{<,E^*})=p_{E}$, in which case $f$ behaves like $-$.

Case 1c. $f(p_{<,E})=p_{E}$ and $f(p_{<,E^*})=p_{E}$. We will use \autoref{containsgamma} (ii) to show that $G$ contains Aut$(\Gamma)$, so let $\bar{a} \in D$. Then there is $\bar{b} \in D$ such that $\bar{a} \cong \bar{b}$ \emph{as digraphs}, and $b_i < b_j$ for all $i<j$. Observe that for all edges $b_ib_j$ in $\bar{b}$, we have that $E(f(b_i),f(b_j)) \leftrightarrow i<j$. By homogeneity of $(D,E)$, there is $g_1 \in G$ such that $g_1(\bar{a})=\bar{b}$. Since $f \in \tmcl(G)$, there is $g_2 \in G$ such that $g_2(\bar{b})=f(\bar{b})$. But now $g=g_2g_1$ satisfies the assumptions of \autoref{containsgamma}, so we conclude that $G$ contains Aut$(\Gamma)$.

The case where $f(p_{<,E})=p_{E^*}$ and $f(p_{<,E^*})=p_{E^*}$ is symmetric to this case.

Case 1d. $f(p_{<,E})=p_{N}$ or $f(p_{<,E^*})=p_{N}$. Without loss suppose the first is true, the latter case is symmetric. We will use \autoref{equalsSymD} (iii) and show that $G=$ Sym$(D)$, so in particular $G$ contains Aut$(\Gamma)$. Let $\bar{a} \in D$ contain an edge $a_ia_j$. By homogeneity of $(D,E)$ there is $g_1 \in G$ such that $g_1(\bar{a}) \cong \bar{a}$ and $g_1(a_i)<g_1(a_j)$. Now let $g_2 \in G$ equal $f$ on $g_1(\bar{a})$. Observe that $g_2$ deletes the edge $g_1(a_i)g_1(a_j)$ (and possibly others too) and we also know that $g_2$ preserves edges. Hence, $g_2g_1(\bar{a})$ contains less edges than $\bar{a}$, so by \autoref{equalsSymD}, we are done.

\underline{Case 2} $f(p_{<,N})=p_E$.

Case 2a. Neither $f(p_{<,E})=p_{N}$ nor $f(p_{<,E^*})=p_{N}$. We will use \autoref{equalsSymD} (iii) and show that $G=$ Sym$(D)$. Let $\bar{a} \in D$ contain an edge $a_ia_j$. By homogeneity of $(D,E)$, map $\bar{a}$ to an isomorphic (as digraphs) tuple $\bar{b}$ where $b_i$, resp. $b_j$, is the least, resp. second least, element of $\bar{b}$. By assumption, $f$ maps $\{b_i,b_j\}$ to an edge but we do not know its direction. This splits into two cases.

Subcase (i). Suppose we have $E(f(b_i),f(b_j))$. Now let $\bar{b}'$ be an ordered digraph which is the same as $\bar{b}$ except that $b_i,b_j$ is changed to a non-edge. Observe that $f(\bar{b}') \cong f(\bar{b})$, because we are in the case where $f(p_{<,N})=p_E$.

But now we are done: we can find mappings in $G$ to get from $\bar{a}$ to $\bar{b}$ to $f(\bar{b})$ to $f(\bar{b}')$ to $\bar{b}'$ (noting that though $f$ may not be invertible, the function in $G$ which agrees with $f$ on $\bar{b}'$ is invertible), and $\bar{b}'$ has less edges than in $\bar{a}$.

Subcase (ii). Suppose we have $E^*(f(b_i,b_j))$. The previous argument does not work as stated, because the edges $f(b_i,b_j)$ and $f(b_i',b_j')$ will not be in the same direction. To fix this, we modify $\bar{b'}$ by swapping $b_i'$ and $b_j'$ with respect to the linear order. Now,  $f(b_i,b_j)$ and $f(b_i',b_j')$ will be in the same direction. Furthermore, because we earlier specified that $b_i$ and $b_j$ should be the two least elements of $\bar{b}$, this swapping only affects the type of the pair $b_ib_j$ - all other pairs' types are unaffected. This ensures that we have $f(\bar{b}') \cong f(\bar{b})$. The rest of the proof continues as in the previous case.

Case 2b. $f(p_{<,E})=p_{N}$ and $f(p_{<,E^*})=p_{N}$. By using the claim, it is easy to see that $f^2$ is a canonical function in $\tmcl(G)$ where $f^2(p_{<,N})=p_N$, $f(p_{<,E})=p_{E}$ and $f(p_{<,E^*})=p_{E}$. Hence, by Case 1c, $G$ contains Aut$(\Gamma)$.

Case 2c. $f(p_{<,E})=p_{N}$ and $f(p_{<,E^*})=p_{E}$. By considering $f^2$, this case is reduced to Case 1d, so $G$ contains Aut$(\Gamma)$.

Case 2d. $f(p_{<,E})=p_{N}$ and $f(p_{<,E^*})=p_{E^*}$. We will use \autoref{equalsSymD} (iii). Let $\bar{a} \in D$ contain an edge $E(a_i,a_j)$. By composing with an element of Aut$(D,E)$ if necessary, we may assume that $a_j$ is the least element, and $a_i$ is the second least element. In particular, we have tp$(a_j,a_i)=p_{<,E^*}$. Let $\bar{b}$ be an ordered digraph such that $f(\bar{b})=\bar{a}$. Now let $\bar{b}'$ be an ordered digraph which is the same as $\bar{b}$ except we swap the position of $b_i'$ and $b_j'$ in the linear order. Now, tp($b_j',b_i')=p_{<,E}$. By choosing $b_i,b_j$ to be the least elements, the types of all the other pairs are unaffected, so $f(\bar{b}')$ is the same digraph as $\bar{a}$ but the edge $a_ia_j$ is replaced by a non-edge. Hence, we are done.

Case 2e. $f(p_{<,E})=p_{E}$ and $f(p_{<,E^*})=p_{N}$. Considering $f^2$ reduces us to Case 2a.

Case 2f. $f(p_{<,E})=p_{E^*}$ and $f(p_{<,E^*})=p_{N}$. Imitate the argument in Case 2d to show that $G=$ Sym$(D)$.

\underline{Case 3} $f(p_{<,N})=p_{E^*}$. This is symmetric to Case 2.
\end{proof}


\subsubsection{Canonical functions from $(D,E,<,\bar{c})$}
We now move on to the general situation where we have added constants $\bar{c} \in D$ to the structure. For convenience, we may as well assume that $c_i<c_j$ for all $i<j$. As is the case for $(D,E)$ (see \autoref{genericdigraph}), the $n$-types of $(D,E,<,\bar{c})$ correspond to the orbits of Aut$(D,E,<,\bar{c})$ acting on the set of $n$-tuples of $D$. As a result we use the concepts of types and orbits interchangeably, and we often abuse notation to provide for a smoother presentation.

Unlike the situation with no constants, this structure is not 1-transitive, i.e., we have more than one orbit. There are two kinds of orbits. The first is a singleton containing one of the constants, e.g. $\{c_1\}$ is an orbit. The second kind consists of infinite orbits, which are necessarily isomorphic to the generic digraph. An infinite orbit is determined by how its elements are related to the $c_i$, e.g., one of the orbits will be $\{x \in D: x<c_1 \wedge \bigwedge_i E(x,c_i)\}$.

In order to describe the 2-types, we extend notation from the previous section.

\textbf{Notation} Let $A,B$ be definable subsets of $D$ and let $\phi_1(x,y),\ldots,\phi_n(x,y)$ be formulas. We let $p_{A,B,\phi_1,\ldots,\phi_n}(x,y)$ denote the (partial) type determined by the formula $x \in A \wedge y \in B \wedge \phi_1(x,y) \wedge \ldots \wedge \phi_n(x,y)$.

Now let $X$ and $Y$ be orbits, $\phi \in \{<,>\}$ and $\psi \in \{E,E^*,N\}$. Then all the 2-types are of the form $p_{X,Y,\phi, \psi}= \{(a,b) \in D: a \in X, b \in Y, \phi(a,b)$ and $\psi(a,b)\}$.\footnote{This is an example of how we are abusing notation and blurring the distinction between types and orbits.}

Our task now is to analyse the possibilities for $f(p_{X,Y,\phi,\psi})$, where $f$ is a canonical function. The analysis is split into cases depending on how the orbits $X$ and $Y$ relate. The first lemma deals with the situation when $X=Y$.

\begin{lem} \label{oneorbit} Let $G$ be a closed supergroup of Aut$(D,E)$, let $f \in \tmcl(G)$ be a canonical function from $(D,E,<,\bar{c})$ and let $X$ be an infinite orbit of Aut$(D,E,\bar{c})$. Then (at least) one of the following holds:
\begin{itemize}
\item $f$ behaves like $id$ on $X$.
\item $f$ behaves like $-$ on $X$.
\item $G$ contains Aut$(\Gamma)$.
\end{itemize}
\end{lem}
\begin{proof} By noting that $(X,E|_X)$ is isomorphic to $(D,E)$, unravelling the definitions will show that this lemma has exactly the same mathematical content as \autoref{noconstants}.
\end{proof}

Next, we look at how $f$ can behave between two infinite orbits. To do this analysis, we need to look at how two infinite orbits can relate to each other with respect to the linear order:

\textbf{Facts and Notation} There are two ways that two infinite orbits $X$ and $Y$ of Aut($D,E,<,\bar{c})$ can relate to each other with respect to the linear order $<$:
\begin{itemize}
\item All of the elements of one orbit, $X$ say, are smaller than all of the elements of $Y$. This is abbreviated by $X<Y$
\item $X$ and $Y$ are interdense: $\forall x<x' \in X, \exists y \in Y$ such that $x<y<x'$ and vice versa.
\end{itemize}

We deal with these two possibilities separately, starting with the case where one orbit is below the other.

\begin{lem} \label{twoorbits} Let $G$ be a closed supergroup of Aut$(D,E)$, let $f \in \tmcl(G)$ be a canonical function from $(D,E,<,\bar{c})$ and let $X$ and $Y$ be infinite orbits of Aut$(D,E,\bar{c})$ such that $X<Y$. Then (at least) one of the following holds: \begin{itemize}
\item $f$ behaves like $id,sw,rot$ or $rot^{-1}$ between $X$ and $Y$.
\item $f$ behaves like $- \circ rot_{(X,Y)}$ on $X \cup Y$.
\item $G$ contains Aut($\Gamma$).
\end{itemize}
\end{lem}

\begin{proof} Let $x_0 \in X$ be fixed. We emphasise now an important feature of this proof, which is that our arguments only depend on how $f$ behaves on $\{x_0\} \cup Y$. This is done intentionally so that these arguments can be used unaltered in later lemmas.

By \autoref{oneorbit}, we may assume that $f$ behaves like $id$ or $-$ on $X$ and $Y$. As mentioned above, the arguments only concern one point $x_0 \in X$, so it does not matter how $f$ behaves on $X$. However, whether $f$ behaves like $id$ or $-$ on $Y$ can make a difference. Fortunately, for most cases the arguments require very little, if any, adjustment, so we assume $f$ acts like $id$ on $Y$. When required, we will explain how to modify the argument if $f$ acts like $-$ on $Y$.

Furthermore, as the arguments are similar to that of \autoref{oneorbit}, the proofs are more sketchy, and we leave the details to the reader.

\underline{Case 1} $f(p_{X,Y,N})=p_{N}$.\footnote{Note that because $X<Y$, $p_{X,Y,\psi}=p_{X,Y,<,\psi}$ for any formula $\psi$}

Case 1a. $f(p_{X,Y,E})=p_{E}$ and $f(p_{X,Y,E^*})=p_{E^*}$. Then $f$ behaves like $id$ between $X$ and $Y$.

Case 1b. $f(p_{X,Y,E})=p_{E^*}$ and $f(p_{X,Y,E^*})=p_{E}$. Then $f$ behaves like $sw$ between $X$ and $Y$.

Case 1c. $f(p_{X,Y,E})=p_{E}$ and $f(p_{X,Y,E^*})=p_{E}$. We will use \autoref{containsgamma} (ii) to show that $G$ contains $\Gamma$. Let $\bar{a}=(a_1,\ldots,a_n) \in D$. We want to show that by using elements of $G$, we can switch the direction of the edges of $\bar{a}$ so they are all pointing in the same direction. We do this by induction on $n$. The base case $n=1$ is trivial so let $n>1$. By the inductive hypothesis, we can assume that for $2\leq i,j \leq n$, if $a_ia_j$ is an edge, then $E(i,j) \leftrightarrow i<j$. By homogeneity, map $a_1$ to $x_0$ and the other $a_i$'s into $Y$. Then applying $f$ switches the edges adjacent to $a_1$ so they are all directed out of $a_1$ and furthermore $f$ does not alter any of the other edges. Thus, the resulting digraph has all edges going in the same direction, as required.

In the case where $f$ behaves like $-$ on $Y$, after applying the induction hypothesis, you first map all of $\bar{a}$ into $Y$, apply $f$ (so switch \emph{all} the edges' directions), map $a_1$ to $x_0$, and apply $f$ again (so we `unswitch' all the edges in $\{a_2,\ldots,a_n\}$).

Case 1d. $f(p_{X,Y,E})=p_{N}$ or $f(p_{X,Y,E^*})=p_{N}$. Given any $\bar{a} \in D$ which contains edges, we use $f$ to delete edges from it, so by \autoref{equalsSymD} (iii), $G=$ Sym$(D)$. See Case 1d of \autoref{oneorbit} for more detail.

\underline{Case 2} $f(p_{X,Y,N})=p_{E}$.

Case 2a. Neither $f(p_{X,Y,E})=p_{N}$ nor $f(p_{X,Y,E^*})=p_{N}$. $G=$ Sym$(D)$, by using the same argument as in Case 2a of \autoref{oneorbit}. Note that the argument would be unaffected if $f$ behaves like $-$ on $X$.

Case 2b. $f(p_{X,Y,E})=p_{N}$ and $f(p_{X,Y,E^*})=p_{N}$. Given any $\bar{a} \in D$ containing edges, we apply $f$ twice in order to delete edges. Thus, $G=$ Sym$(D)$ by part (iii) of \autoref{equalsSymD}.

Case 2c. $f(p_{X,Y,E})=p_{N}$ and $f(p_{X,Y,E^*})=p_{E}$. Same as Case 2b.

Case 2d. $f(p_{X,Y,E})=p_{N}$ and $f(p_{X,Y,E^*})=p_{E^*}$. 

Subcase (i) $f$ acts like $id$ on $Y$. We show that $G=$ Sym$(D)$. The idea is the same to that of Case 2d in \autoref{oneorbit}. Let $\bar{b} \in D$ contain an edge $E(b_i,b_j)$. Let $\bar{b}_1$ be obtained by mapping $b_i$ to $x_0$, applying $f$, then mapping $b_j$ to $x_0$ and applying $f$ again. Note that we have $E^*(b_{1,i},b_{1,j})$. Let $\bar{b}_2$ be obtained from $\bar{b}$ in the same way, except we map $b_j$ to $x_0$ first, and then $b_i$ second.  In this case, we have $N(b_{2,i},b_{2,j})$.  Furthermore, $\bar{b}_1$ and $\bar{b}_2$ are otherwise the same. Now suppose $\bar{a}$ is given. Find a $\bar{b}$ such that its corresponding $\bar{b}_1$ is isomorphic to $\bar{a}$. Then, we can get from $\bar{a}$ to $\bar{b_1}$ to $\bar{b_2}$, i.e., we can delete an edge from $\bar{a}$. Thus, $G=$ Sym$(D)$.

Subcase (ii) $f$ acts like $-$ on $Y$. Now, if $f$ behaves like $id$ on $X$, the subcase above shows that $G=$ Sym$(D)$. Hence, we are left with case where $f$ behaves like $id$ on both $X$ and $Y$. But then $f$ behaves like $- \circ rot_{(X,Y)}$ on $X \cup Y$.

Case 2e. $f(p_{X,Y,E})=p_{E}$ and $f(p_{X,Y,E^*})=p_{N}$. Considering $f^2$ reduces us to Case 2a.

Case 2f. $f(p_{X,Y,E})=p_{E^*}$ and $f(p_{X,Y,E^*})=p_{E}$. Then $f$ behaves like $rot$ between $X$ and $Y$. Note that if $f$ behaves like $-$ on $Y$ in this case, you can show that $G=$ Sym$(D)$ by considering $- \circ f \in \tmcl(G)$.
\end{proof}

As mentioned at the start of the proof, what was relevant is how $f$ behaved on $\{x_0\} \cup Y$.  More specifically, what was sufficient to make these arguments work was the following: For all finite digraphs $\bar{a} \in D$ and all points $a \in \bar{a}$, we can find a copy of $\bar{a}$ in $\{x_0\} \cup Y$ such that $\bar{a} \cap \{x_0\} = \{a\}$. 

This condition is satisfied in the remaining situations that need to be analysed, so their corresponding results are immediate corollaries of \autoref{twoorbits}. The statement for interdense orbits has to be modified and will perhaps appear confusing. Clarification will be provided after the statement.

\begin{cor}
Let $G$ be a closed supergroup of Aut$(D,E)$, let $f \in \tmcl(G)$ be a canonical function from $(D,E,<,\bar{c})$ and let $X$ and $Y$ be interdense infinite orbits of Aut$(D,E,\bar{c})$. Further suppose that $G$ does not contain Aut$(\Gamma)$. Then both of the following hold: \begin{itemize}
\item $f$ behaves like $id,sw,rot$ or $rot^{-1}$ between increasing tuples from $X$ to $Y$, or, $f$ behaves like $- \circ rot$ on the set of increasing tuples from $X$ to $Y$.
\item $f$ behaves like $id,sw,rot$ or $rot^{-1}$ between decreasing tuples from $X$ to $Y$, or, $f$ behaves like $- \circ rot$ on the set of decreasing tuples from $X$ to $Y$.
\end{itemize}
\end{cor}

For example, $f$ behaves like $sw$ between increasing tuples from $X$ to $Y$ means that $f(p_{X,Y,<,E})=p_{E^*}, (p_{X,Y,<,E^*})=p_{E}$ and $(p_{X,Y,<,N})=p_{N}$.  Another example is that when we say $f$ behaves like $- \circ rot$ on the set of decreasing tuples from $X$ to $Y$, we mean that $f$ behaves like $-$ on $X$ and $Y$, $f(p_{X,Y,>,E})=p_{N}, (p_{X,Y,>,E^*})=p_{E^*}$ and $(p_{X,Y,>,N})=p_{E}$.

The final situation is to look at how $f$ can behave between one of the constants, $c$ say, and the infinite orbits.

\begin{Def}: Let $c$ be one of the named constants of $(D,E,<,\bar{c})$ and let $X_1,X_2$ and $X_3$ be infinite orbits. If it is the case that we have outward edges from $c$ to $X_1$, inward edges from $c$ to $X_2$ and non-edges between $c$ and $X_3$, we called the triple $\bar{X}=(X_1,X_2,X_3)$ a $c$-generic triple. 
\end{Def}

The reason for introducing this definition is that there is nothing to be analysed about how $f$ behaves between $c$ and a \emph{single} orbit $X$. It is only useful to ask how $f$ behaves between $c$ and several infinite orbits, in particular, a $c$-generic triple. Note that if $\bar{X}$ is a $c$-generic triple, then $c \cup X_1 \cup X_2 \cup X_3$ is isomorphic to the generic digraph.

\begin{cor} Let $G$ be a closed supergroup of Aut$(D,E)$, let $f \in \tmcl(G)$ be a canonical function from $(D,E,<,\bar{c})$, let $c$ be one of the named constants and let $\bar{X}$ be a $c$-generic triple. Then (at least) one of the following holds:
\begin{itemize}
\item $f$ behaves like $id,sw,rot$ or $rot^{-1}$ between $c$ and $\bigcup \bar{X}$.
\item $f$ behaves like $-\circ rot_{(c,\bar{X})}$ on $\{c\} \cup \bigcup \bar{X}$.
\item $G$ contains Aut($\Gamma$).
\end{itemize}
\end{cor}


\subsection{Using canonical functions}
With this analysis, we are now in a position to prove the remaining lemmas.

\begin{proof}[Proof of \autoref{threeregions}] We recall that we want to show that if $G$ is a reduct of $(D,E)$, then $G$ either contains Aut$(\Gamma)$, is contained in Aut$(\Gamma)$, or contains $\langle rot \rangle$.  Suppose none of these are true - we will derive a contradiction.

$G$ is not contained in Aut$(\Gamma)$, which means that $G$ does not preserve non-edges. Hence, there is $f \in G$ and an edge $c_1c_2 \in D$ such that $f(c_1c_2)$ is a non-edge. We apply \autoref{blackbox} to obtain a canonical $g: (D,E,<,c_1,c_2) \to (D,E)$ which agrees with $f$ on $c_1$ and $c_2$.

By \autoref{oneorbit}, for any infinite orbit $X$, $g$ behaves like $id$ or $-$ on it - otherwise, $G$ would contain Aut($\Gamma$), contradicting our assumptions. By \autoref{twoorbits} and its corollaries, we have that $g$ behaves like $id$ or $sw$ between orbits - otherwise, $G$ would contain either Aut($\Gamma$) or $\langle rot \rangle$, contradicting our assumptions.

But now we have a function $g \in \tmcl(G)$ which deletes an edge (namely,$c_1c_2$) and maps all non-edges to non-edges. By imitating the proof of part (iv) of \autoref{equalsSymD}, we conclude that $G$ equals Sym$(D)$, contradicting that $G$ does not contain Aut($\Gamma$). 
\end{proof}

\vspace{5mm}
\begin{proof}[Proof of \autoref{region2}] By \autoref{region2a} and \autoref{region2b}, it remains to be proved that if $\langle sw,- \rangle < G \leq$ Aut($\Gamma)$ is a closed group, then $G=$ Aut$(\Gamma)$.

So let $G$ be such a closed group and suppose, for contradiction, that $G \neq$ Aut$(\Gamma)$. Since $\langle sw,- \rangle < G$, $G$ does not preserve $P_{sw,w}$, there exists $f \in G$ and $\bar{c} \in D$ such that $P_{sw,w}(\bar{c})$ and $\neg P_{sw,w}(f(\bar{c}))$. Now use \autoref{blackbox} to obtain $g \in \tmcl(G)$ which is canonical from $(D,E,<,\bar{c})$ and which agrees with $f$ on $\bar{c}$.

As in the previous proof, we use $\autoref{oneorbit}$ to conclude that for any infinite orbit $X$, $g$ behaves like $id$ or $-$ on $X$, and we use $\autoref{twoorbits}$ and its corollaries to conclude that $g$ behaves like $id$ or $sw$ between orbits.

\textbf{Claim 1.} $g$ behaves like $id$ on all infinite orbits, or, $g$ behaves like $-$ on all infinite orbits.

\emph{Proof of Claim 1.} Suppose not, so there exists infinite orbits $X$ and $Y$ such that $g$ behaves like $id$ on $X$ and like $-$ on $Y$. There are now two cases. The first case is if $g$ behaves like $id$ between $X$ and $Y$. In this case, by imitating the proof of part (iii) of \autoref{containsgamma}, we conclude that $G \geq$ Aut$(\Gamma)$ - contradiction. The second case is if $g$ behaves like $sw$ between $X$ and $Y$. This reduces to the previous case by considering $- \circ g$: $-\circ g$ behaves like $-$ on $X$, like $id$ on $Y$ and like $id$ between $X$ and $Y$.  Thus, we always reach a contradiction, proving the claim.

In fact, by considering $- \circ g$ if necessary, we may now assume that $g$ behaves like $id$ on all infinite orbits.

Enumerate the (finite number of) infinite orbits as $X_1,X_2,X_3,\ldots$.

\textbf{Claim 2.} We may assume $g$ behaves like $id$ between all pairs of orbits from $X_1,X_2,X_3$.

\emph{Proof of Claim 2.} If $g$ behaves like $id$ between all three orbits, we are done. If $g$ behaves like $sw$ between precisely two of the pairs - without loss $g$ behaves like $sw$ between $X_1$ and $X_2$, and between $X_1$ and $X_3$ - then by switching about $X_1$ we may assume $g$ behaves like $id$, so again we are done.  If $g$ behaves like $sw$ between precisely one pair of infinite orbits, then by imitating the proof of \autoref{containsgamma} (iii), we get that $G \geq$ Aut$(\Gamma)$, which is a contradiction. The final possibility is that $g$ behaves like $sw$ between all pairs: this reduces to the third case by switching about $X_1$, so we again get a contradiction.  Thus, we have proved the claim.

\textbf{Claim 3.} We may assume $g$ behaves like $id$ between all infinite orbits.

\emph{Proof of Claim 3.} First consider how $g$ behaves between $X_4$ and the first three $X_i$'s. If $g$ behaves like $id$ between $X_4$ and all the previous $X_i$'s, we move on. If $g$ behaves like $sw$ between $X_4$ and all the previous $X_i$'s, we switch about $X_4$, reducing to the first case. The last case is if, without loss, $g$ behaves like $id$ between $X_4$ and $X_1$ and like $sw$ between $X_4$ and $X_2$. But this is exactly the same as the contradictory case in the proof Claim 2, so this is not possible. Hence, we have shown that $g$ must behave like $id$ between all the pairs in $X_1,\ldots,X_4$.

One then moves on to $X_5$ and repeats this argument to show that we may assume $g$ behaves like $id$ between $X_1,\ldots X_5$. Continuing in this fashion proves the claim.

\textbf{Claim 4.} We may assume that $g$ behaves like $id$ between $\bar{c}$ and the infinite orbits.

Consider $c_1 \in \bar{c}$, and let $\bar{X}$ and $\bar{Y}$ be $c_1$-generic.  Suppose that $g$ behaves like $id$, resp. $sw$ between $c$ and $\bigcup \bar{X}$, resp. $\bigcup \bar{Y}$.  Then, for any finite digraph $A$ and edge $aa' \in A$, we can map $aa'$ to an edge between $c_1$ and $\bar{Y}$ and map the remaining vertices of $A$ into $\bar{X}$. Then applying $g$ will have the effect of switching precisely the single edge $aa'$. Thus, by imitating the argument of \autoref{containsgamma} (iii), we get that $G \geq$ Aut$(\Gamma)$, a contradiction.  Hence, it must be the case $g$ behaves like $id$ between $c_1$ and all the $c_1$-generic triples, or, $g$ behaves like $sw$ between $c_1$ and the triples.  But in the latter case, we can apply $sw_{c_1}$ to reduce to the former case.  

Repeating this for the other $c_i$ will complete the proof of the claim.

Observe that all the manipulations we (may have) used on $g$ have been applications of $-$ or $sw$.  This ensures that $g(\bar{c} \notin P_{sw,w}$. In particular, there is at least one edge in $\bar{c}$ whose direction $g$ switches.  Combining this observation with all the claims tells us that we are in the situation of \autoref{containsgamma} (iii): $g$ behaves like $id$ everywhere except on the finite set $\bar{c}$. Hence, $G \geq$ Aut$(\Gamma)$, giving us a contradiction, thus completing the proof.
\end{proof}

\begin{proof}[Proof of \autoref{region3}] We recall the statement of the lemma: If $G$ contains $\langle rot \rangle$, then $G$ equals $\langle rot \rangle, \langle -,rot \rangle$ or Sym$(D)$.

To prove the statement, it suffices to prove the following: \begin{enumerate}[(i)]
\item If $G > \langle rot \rangle$ and $G \not\geq \langle -,rot \rangle$, then $G=$ Sym$(D)$.
\item If $G > \langle - ,rot \rangle$, then $G=$ Sym$(D)$.
\end{enumerate}

Before continuing, recall that $\langle sw,rot \rangle =$ Sym$(D)$. Hence, we may assume in all that follows that $sw \notin G$.

(i) Suppose $G > \langle rot \rangle$ and $G \not\geq \langle -,rot \rangle$. The latter assumption implies that $- \notin G$. Then there exists $\bar{c}$ and $f \in G$ which witness the fact that $G$ does not preserve $P_{rot,1}$. Then use \autoref{blackbox} to obtain a canonical $g: (D,E,<,\bar{c}) \to (D,E)$ which agrees with $f$ on $\bar{c}$. 

By \autoref{oneorbit}, $g$ behaves like $id$ on all infinite orbits, as otherwise $G$ contains $sw$ or $-$. Similarly, by \autoref{twoorbits} and its corollaries, we have that $g$ behaves like $id$ or $rot$ between orbits.

We proceed in a similar fashion to the proof of \autoref{region2}.

\textbf{Claim 1.} We may assume that $g$ behaves like the $id$ between all infinite orbits.

\emph{Proof of Claim 1.} Let $X_1,X_2,\ldots$ enumerate the infinite orbits.  If $g$ behaves like $id$ between $X_1$ and $X_2$ then move on. Otherwise, $g$ behaves like $rot$ or $rot^{-1}$ between $X_1$ and $X_2$. Hence, by composing with a rotation about $X_1$ or $X_2$ as appropriate, we can assume $g$ behaves like $id$ between $X_1$ and $X_2$.

Now consider $X_3$. Again by composing with a rotation if necessary, we may assume that $g$ behaves like $id$ between $X_1$ and $X_3$. Suppose that $g$ does not behave like $id$ between $X_2$ and $X_3$ - we will show that $G$ must equal Sym$(D)$.  Without loss, $g$ behaves like $rot$ between $X_2$ and $X_3$. Given any finite digraph $A$ and an edge $aa'$ in $A$, we can find a copy of $A$ in $D$ such that $aa'$ is an inward edge from $X_2$ to $X_3$ and such that the other vertices of $A$ all lie in $X_1$. Applying $g$ to this copy results in the edge $aa'$ being deleted, with the rest of $A$ being the same.  Hence, by \autoref{equalsSymD} (iii), we conclude that $G=$ Sym$(D)$.  Therefore, we may assume that $g$ behaves like $id$ between $X_1,X_2$ and $X_3$.

We then consider $X_4$. Using an identical argument, we can show that if $g$ does not behave like $id$ between $X_4$ and the other orbits, then $G=$ Sym$(D)$, so we may assume $g$ behaves like $id$.  Continuing in this fashion proves the claim.

\textbf{Claim 2.} We may assume that $g$ behaves like $id$ between $\bar{c}$ and all the infinite orbits.

\emph{Proof of Claim 2.} First work with $c_1$. Let $\bar{X}$ be a $c_1$-generic triple. We know that $g$ behaves like $id$, $rot$ or $rot^{-1}$ between $c$ and $\bar{X}$, so by composing with a rotation if necessary, we may assume that $g$ behaves like $id$.  Now consider another $c_1$-generic triple $\bar{Y}$. If $g$ does not behave like $id$ between $c_1$ and $\bar{Y}$, then we use the same argument as in the proof of Claim 1 to show that $G$ equals Sym$(D)$. So we may assume that $g$ behaves like $id$ between $c_1$ and all infinite orbits.

Repeating this for $c_2$ and $c_3$ completes the proof of the claim.

Because all the possible modifications of $g$ were compositions with rotations, we still have $\neg P_{rot,1}(g(\bar{c}))$.  But that means we have $g$ acting like $id$ everywhere except on $\bar{c}$.  Hence, $g$ must either switch an edge or delete an edge in $\bar{c}$, but we do not know which. In either case, using \autoref{containsgamma} or \autoref{equalsSymD} as appropriate, we get that $G \geq$ Aut$(\Gamma)$, so $G$ contains $sw$, so $G=$ Sym$(D)$. This completes the proof of (i).

(ii) Suppose $G > \langle rot,- \rangle$. Then there exists $\bar{c}$ and $f \in G$ which witness the fact that $G$ does not preserve $P_{rot,w}$. Then use \autoref{blackbox} to obtain a canonical $g: (D,E,<,\bar{c}) \to (D,E)$ which agrees with $f$ on $\bar{c}$. 

By \autoref{oneorbit}, $g$ behaves like $id$ or $-$ on all infinite orbits, and by \autoref{twoorbits} and its corollaries, we have that $g$ behaves like $id$ or $rot$ between orbits, or like $-\circ rot$ on the union of the two orbits.

\textbf{Claim 1'.} We may assume that $g$ acts like $id$ on all infinite orbits or like $-$ on all infinite orbits.

\emph{Proof of Claim 1'.} Suppose not. so there are infinite orbits $X$ and $Y$ such that $g$ acts like $id$ on $X$ and $-$ on $Y$. There are two options for how $g$ behaves between $X$ and $Y$, like $id$ or like $rot$. If $g$ behaves like $id$, then we imitate the idea in \autoref{containsgamma} (iii) to show that $G \geq$ Aut$(\Gamma)$, so then $G$ must equal Sym$(D)$. If $g$ behaves like $rot$ between $X$ and $Y$, just compose with $rot_{X}$ to reduce to the first option. This completes the proof of the claim.

Now composing with $-$ if necessary, we may assume that $g$ behaves like $id$ on all infinite orbits.

What are the possible ways $g$ can behave between an infinite and another orbit?  There are three options: like $id$, like $rot$, or, like $-\circ rot$.  This last option is possible as it may have been the case that $g$ originally behaved like $rot$ between the two orbits, so after applying $-$ we get that it behaves like $-\circ rot$. But what does it mean to behave like $- \circ rot$ between two orbits?  It means that you swap outward edges with non-edges, while preserving inward edges the same. This means that we would be in Case 2di of \autoref{twoorbits}, where we showed that $G$ must then equal Sym$(D)$.

Hence, we may assume that $g$ acts like $id$ on all infinite orbits, and that $g$ behaves like $rot$ or $id$ between an infinite orbit and any other orbit.  But this is exactly the same situation as in part (i) of this proof, so we just repeat the argument. This completes the proof.
\end{proof}


\section{Summary and Open Questions}
We summarise the structure of the proof of the main theorem, \autoref{maintheorem}, which states that $\mathcal{L}$ is the lattice of the reducts of the generic digraph.  The first task is to show that $\mathcal{L}$ is a sublattice of the reducts of the generic digraph, which was done in \autoref{sublattice}.

The second task is to show that $\mathcal{L}$ contains all the reducts. By \autoref{threeregions}, which was proved using canonical functions at the start of Section 5.3, the task is split up into three regions of $\mathcal{L}$: The region above Aut$(\Gamma)$, the region below Aut$(\Gamma)$, and the rest.  The region above Aut$(\Gamma)$ is immediately dealt with by Thomas' classification of the reducts of $\Gamma$.  The proof of the region below Aut$(\Gamma)$, \autoref{region2}, has two parts. The first part is in Section 5.1, where we use the function $T(G)$ and the classification for the random tournament, and the second part is in Section 5.3. The final region, \autoref{region3}, is proved using canonical functions at the end of Section 5.3.

We end by stating some problems of interest in this area. There is the obvious task of determining the reducts of your favourite structure(s), but some more specific questions are:
\begin{itemize}
\item (Thomas' Conjecture): If a structure is homogeneous in a finite relational language, then it only has finitely many reducts.
\item Which lattices can be realised as the lattice of reducts of some structure?
\item Is there always a maximal closed group between a closed group $G$ and Sym$(M)$ (where $M$ is countable)?
\end{itemize}

The answer to the first question may be related to a question in structural Ramsey theory: Given a homogeneous structure, can you finitely extend its language so that the structure becomes Ramsey? For example, it may be easier to prove the conjecture is true for Ramsey structures, and this may be sufficient to prove the full conjecture.  Alternatively, a counterexample for one question may lead to a counterexample of the other.

Another angle on the second question could be to consider whether there is any relationship between structures which have the same lattice of reducts. For example, it is curious that $(\mathbb{Q},<)$, the random graph, the random tournament and the generic partial order have the same 5-element lattice as their lattice of reducts.

For clarification of the third question, we say that a closed group $F<$ Sym$(M)$ is maximal if there are no closed groups $F'$ such that $F < F' <$ Sym$(M)$. To find a counterexample to this question, it is sufficient to find an $\aleph_0$-categorical countable structure such that all of its non-trivial reducts have infinitely many reducts.  We remark that such a structure without the condition of being $\aleph_0$-categorical is known: $(\mathbb{Z}, \{(x,y): |x-y|=1\})$.  The relations definable in this structure are analysed in \cite{ss12}, and it follows that all of its non-trivial reducts have infinitely many reducts.

\bibliographystyle{alpha}
\bibliography{bibliography}

\end{document}